\documentclass[11pt, reqno]{amsart}
\usepackage{amsmath, amssymb, amsthm, graphicx,marvosym}
\usepackage{cancel}
\usepackage[nobysame]{amsrefs}[11pts]

\usepackage{color}

\newtheorem{theorem}{Theorem}[section]
\newtheorem{definition}{Definition}[section]
\newtheorem{lemma}{Lemma}[section]
\newtheorem{remark}{Remark}[section]

\numberwithin{equation}{section}
\numberwithin{figure}{section}

\makeatletter \@addtoreset{equation}{section} \makeatother

\def\tilde{\widetilde}

\newcommand{\beq}{\begin{equation}}
\newcommand{\eeq}{\end{equation}}

\begin{document}
\title[Navier-Stokes equations]{On the breakdown   of regular solutions with finite energy for 3D   degenerate compressible  Navier-Stokes equations}



\author{Shengguo Zhu}
\address[Shengguo   Zhu]{Mathematical Institute, University of Oxford,  Oxford OX2 6GG, UK.}
\email{\tt zhus@maths.ox.ac.uk}

\begin{abstract}

In this paper, the three-dimensional (3D) isentropic compressible Navier-Stokes equations with degenerate viscosities (\textbf{ICND}) is considered in both the whole space and the periodic domain. First, for the corresponding Cauchy problem, when shear and bulk viscosity coefficients are both given as a constant multiple of the density’s power  ($\rho^\delta$ with $0<\delta<1$), based on some elaborate analysis of this system’s intrinsic singular structures, we show that the $L^\infty$ norm of the deformation tensor $D(u)$ and the $L^6$ norm of $\nabla \rho^{\delta-1}$ control the possible breakdown   of regular solutions with far field vacuum. This conclusion means that if a solution with far field vacuum of the \textbf{ICND} system is initially regular and loses its regularity at some later time, then the formation of singularity must be caused by losing the bound of  $D(u)$ or  $\nabla \rho^{\delta-1}$ as the critical time approaches. Second, under the additional assumption that the shear and second viscosities (respectively $\mu(\rho)$ and $\lambda(\rho)$) satisfy the BD relation $\lambda(\rho)=2(\mu'(\rho)\rho-\mu(\rho))$, if we consider the corresponding problem in some periodic domain and the initial density is away from the vacuum, it can be proved that the possible breakdown of classical solutions can be controlled only by the $L^\infty$ norm of  $D(u)$.  It is worth pointing out that, except the conclusions mentioned above, another purpose of the current paper is to show how to understand the intrinsic singular structures of the fluid system considered now, and then how to develop the corresponding nonlinear energy estimates in the specially designed energy space with singular weights  for the unique  regular solution with finite energy.\end{abstract}

\date{Oct. 16th, 2019}
\subjclass[2010]{35Q30, 35M30, 76N99, 35B44, 35A09} \keywords{Compressible Navier-Stokes equations, three dimensions,  far field vacuum, degenerate viscosity,  intrinsic singular structure,  classical solutions, finite energy, breakdown}.

\maketitle

\tableofcontents

\section{Introduction}

The time evolution of the mass density $\rho\geq 0$ and the velocity $u=\left(u^{(1)},u^{(2)},u^{(3)}\right)^\top$ $\in \mathbb{R}^3$ of a general viscous  isentropic
compressible  fluid occupying a spatial domain $\Omega\subset \mathbb{R}^3$ is governed by the following isentropic  compressible  Navier-Stokes equations:
\begin{equation}
\label{eq:1.1}
\begin{cases}
\rho_t+\text{div}(\rho u)=0,\\[4pt]
(\rho u)_t+\text{div}(\rho u\otimes u)
  +\nabla
   P =\text{div} \mathbb{T}.
\end{cases}
\end{equation}
Here, $x=(x_1,x_2,x_3)\in \Omega$, $t\geq 0$ are the space and time variables, respectively. For the polytropic gases, the constitutive relation is given by
\begin{equation}
\label{eq:1.2}
P=A\rho^{\gamma}, \quad A>0,\quad  \gamma> 1,
\end{equation}
where $A$ is  an entropy  constant and  $\gamma$ is the adiabatic exponent. $\mathbb{T}$ denotes the viscous stress tensor with the  form
\begin{equation}
\label{eq:1.3}
\begin{split}
&\mathbb{T}=2\mu(\rho)D(u)+\lambda(\rho)\text{div}u\,\mathbb{I}_3,
\end{split}
\end{equation}
 where $D(u)=\frac{1}{2}\big(\nabla u+(\nabla u)^\top\big)$ is the deformation tensor,    $\mathbb{I}_3$ is the $3\times 3$ identity matrix,
\begin{equation}
\label{fandan}
\mu(\rho)=\alpha  \rho^\delta,\quad \lambda(\rho)=\beta  \rho^\delta,
\end{equation}
for some  constant $\delta\geq 0$,
 $\mu(\rho)$ is the shear viscosity coefficient, $\lambda(\rho)+\frac{2}{3}\mu(\rho)$ is the bulk viscosity coefficient,  $\alpha$ and $\beta$ are both constants satisfying
 \begin{equation}\label{10000}\alpha>0 \quad \text{and} \quad   \alpha+\beta\geq 0.
 \end{equation}

In the current paper, assuming $0<\delta<1$, we consider the unique smooth solution
$(\rho,u)$  with finite energy to the following two types of problems:
\begin{itemize}
\item Cauchy problem ($\Omega=\mathbb{R}^3$) for \eqref{eq:1.1}-\eqref{10000} with the following initial data and
far field behavior:
\begin{align}
&(\rho,u)|_{t=0}=(\rho_0(x)> 0,  u_0(x)) \ \ \ \ \ \ \  \ \ \   \  \ \text{for} \quad  x\in \Omega,\label{initial}\\[2pt]
&(\rho,u)(t,x)\rightarrow  (0,0) \quad  \text{as}\ \ \  |x|\rightarrow \infty \quad \ \  \  \text{for} \quad \ t\geq 0.  \label{far}
\end{align}
\item  Periodic problem ($\Omega=\mathbb{T}^3$)  for \eqref{eq:1.1}-\eqref{10000} with the  initial data \eqref{initial}, where  $\mathbb{T}^3$ is the three-dimensional torus.
\end{itemize}

In the theory of gas dynamics, the \textbf{CNS} can be derived  from the Boltzmann equations through the Chapman-Enskog expansion, cf. Chapman-Cowling \cite{chap} and Li-Qin  \cite{tlt}. Under some proper physical assumptions,  the viscosity coefficients and heat conductivity coefficient $\kappa$ are not constants but functions of the absolute temperature $\theta$ such as:
\begin{equation}
\label{eq:1.5g}
\begin{split}
\mu(\theta)=&a_1 \theta^{\frac{1}{2}}F(\theta),\quad  \lambda(\theta)=a_2 \theta^{\frac{1}{2}}F(\theta), \quad \kappa(\theta)=a_3 \theta^{\frac{1}{2}}F(\theta)
\end{split}
\end{equation}
for some   constants $a_i$ $(i=1,2,3)$ (see \cite{chap}).  Actually for the cut-off inverse power force models, if the intermolecular potential varies as $r^{-a}$,
 where $ r$ is intermolecular distance,  then in (\ref{eq:1.5g}):
 $$F(\theta)=\theta^{b}\quad  \text{with}\quad   b=\frac{2}{a} \in [0,+\infty).$$
  In particular  (see \S 10 of \cite{chap}), for  ionized gas,
 $$a=1\quad \text{and} \quad b=2;$$
  for Maxwellian molecules,
  $$a=4\quad \text{and} \quad b=\frac{1}{2};$$  
while for rigid elastic spherical molecules,
$$a=\infty\quad \text{and} \quad b=0.$$

According to Liu-Xin-Yang \cite{taiping}, for  isentropic and polytropic fluids , such a dependence is inherited through the laws of Boyle and Gay-Lussac:
$$
P=R\rho \theta=A\rho^\gamma,  \quad \text{for \ \ constant} \quad R>0,
$$
i.e., $\theta=AR^{-1}\rho^{\gamma-1}$,  and one finds that the viscosity coefficients are functions of the density of the form $(\ref{fandan})$. Generally, for most of the physical processes, $\gamma \in (1,3)$, which implies that for rigid elastic spherical molecules,  $\delta\in (0,1)$, which is  exactly the case which  we are going to study. Actually, the similar assumption that viscosity coefficients depend on the density can be seen in a lot of fluid models, such as Korteweg system, shallow water equations, lake equations and quantum Navier-Stokes system (see \cite{bd2, bd, BN2,Gent, ansgar, decayd, lions, Mar, vy}).

Throughout this paper, we adopt the following simplified notations, most of them are for the standard homogeneous and inhomogeneous Sobolev spaces:
\begin{equation*}\begin{split}
\displaystyle
 & \|f\|_s=\|f\|_{H^s(\Omega)},\quad |f|_p=\|f\|_{L^p(\Omega)},\quad \|f\|_{m,p}=\|f\|_{W^{m,p}(\Omega)},\\[6pt]
  \displaystyle  
 &  |f|_{C^k}=\|f\|_{C^k(\Omega)},  \quad \|f\|_{XY(t)}=\| f\|_{X([0,t]; Y(\Omega))},\\[6pt]
 \displaystyle 
 & D^{k,r}=\{f\in L^1_{loc}(\Omega): |f|_{D^{k,r}}=|\nabla^kf|_{r}<+\infty\},\quad D^k=D^{k,2},  \\[6pt]
 \displaystyle 
  & D^{1}=\{f\in L^6(\Omega):  |f|_{D^1}= |\nabla f|_{2}<\infty\},\quad |f|_{D^1}=\|f\|_{D^1(\Omega)},\\[6pt]
  \displaystyle  
 & \|f\|_{X_1 \cap X_2}=\|f\|_{X_1}+\|f\|_{X_2}, \quad  \int_{\Omega}  f \text{d}x  =\int f, \\[6pt]
 & X([0,T]; Y(\Omega))= X([0,T]; Y),\quad \|(f,g,h)\|_X=\|f\|_{X}+\|g\|_{X}+\|h\|_{X}.\end{split}
\end{equation*}
We will clearly indicate that $\Omega=\mathbb{R}^3$ or $\mathbb{T}^3$ where the above notations are used.
 A detailed study of homogeneous Sobolev spaces  can be found in \cite{gandi}.

Because the  momentum equations $(\ref{eq:1.1})_2$ is a double degenerate system when the density loses its positive lower bound, i.e., 
$$
\displaystyle
 \underbrace{\rho(u_t+u\cdot \nabla u)}_{Degenerate   \ time \ evolution \ operator}+\nabla P= \underbrace{\text{div}(2\mu(\rho)D(u)+\lambda(\rho)\text{div}u \mathbb{I}_3)}_{Degenerate\ elliptic \ operator},
$$
usually, it is very hard to control the behavior of the fluids velocity $u$ near the vacuum.   Moreover,
 it should be pointed out here that unlike the case of  constant viscosities, the elliptic operator   $\text{div} \mathbb{T}$ not only loses strong  regularizing effect on solutions,   but also  cause some troubles in the high order   regularity estimates for the velocity.  For example, in order to establish some uniform a priori estimates independent of the lower bound of density   in $H^3$ space, we need to  handle the extra nonlinear terms such as
$$\text{div}\big( \nabla^k\rho^\delta Q(u)\big) \ \ \text{for} \ \  Q(u)=\alpha(\nabla u+(\nabla u)^\top)+\beta\mathtt{div}u\mathbb{I}_3,$$
where $k=1, 2, 3$. Therefore,  many attentions need to be  paid in order to control these strong nonlinearities, especially for considering the related problems with vacuum state   for arbitrarily large time.

For the cases  $\delta\in (0,\infty)$,  if  $\rho>0$, $(\ref{eq:1.1})_2$ can be formally rewritten as
\begin{equation}\label{qiyi}
\begin{split}
u_t+u\cdot\nabla u +\frac{A\gamma}{\gamma-1}\nabla\rho^{\gamma-1}+\rho^{\delta-1} Lu=\psi \cdot  Q(u),
\end{split}
\end{equation}
where the quantities $\psi$ and $Lu$ are given by
\begin{equation}\label{operatordefinition}
\begin{split}
\psi \triangleq & \nabla \log \rho \quad \text{when}\quad \delta=1;\\
\psi \triangleq & \frac{\delta}{\delta-1} \nabla \rho^{\delta-1}\quad \text{when}\quad \delta \in (0,1)\cup (1,\infty);\\
Lu\triangleq &-\alpha \triangle u-(\alpha+\beta)\nabla \text{div}u.
\end{split}
\end{equation}
When $\delta=1$, from  (\ref{qiyi})-(\ref{operatordefinition}),  the degeneracies of the time evolution and  viscosities on $u$ caused by the  vacuum have  been transferred to the possible singularity of the  term  $\nabla \log \rho$, which actually  can be controlled by a symmetric hyperbolic system with a source term  $\nabla \text{div}u$ in  Li-Pan-Zhu \cite{sz3}. Then via establishing a uniform a priori estimates in $L^6\cap D^1\cap D^2$  for $\nabla \log \rho$, the existence of two-dimensional (2D) local classical solution with far field vacuum   to  (\ref{eq:1.1}) has been obtained in \cite{sz3}, which also applies to the 2D shallow water equations.
When $\delta> 1$,  (\ref{qiyi})-(\ref{operatordefinition}) imply that actually the velocity $u$ can be governed by a nonlinear degenerate  parabolic system without singularity   near the vacuum region.
Based on this observation, by using some hyperbolic  approach which bridges the parabolic system (\ref{qiyi}) when $\rho>0$ and the hyperbolic one  
$$u_t+u\cdot\nabla u=0 \quad \text{when} \quad  \rho=0,$$   the existence of 3D local classical solutions with vacuum to  (\ref{eq:1.1}) was established in  Li-Pan-Zhu \cite{sz333}.
The corresponding global well-posedness in some homogeneous Sobolev spaces  has been established  by Xin-Zhu \cite{zz} under some initial smallness  assumptions. Moreover, under the well-known B-D relation for viscosities:
\begin{equation}\label{bd}
\lambda(\rho)=2(\mu'(\rho)\rho-\mu(\rho)),
\end{equation}
which is introduced by Bresch-Desjardins  and their collaborators in \cite{bd6,bd7, bd2,bd}, recently
the global existences of the multi-dimensional weak solutions with finite energy for some kinds of $\mu(\rho)$ have been given  by Li-Xin \cite{lz} and  Vasseur-Yu \cite{vayu}.

However,  the  approaches used in \cite{sz3,sz333,zz, zhuthesis} for establishing the existence of the unique regular solution  fail to apply to the   case  $\delta\in (0,1)$. Indeed, when vacuum appears only at far fields, the velocity field $u$ is still     governed by  the  quasi-linear parabolic system (\ref{qiyi})-(\ref{operatordefinition}).  Yet,  some new essential  difficulties arise compared with the case $\delta\geq 1$:
\begin{enumerate}
\item
first, the source term contains a stronger singularity as:
$$\nabla \rho^{\delta-1}=(\delta-1)\rho^{\delta-1}\nabla \log \rho,$$
 whose behavior will become more singular than that of $\nabla \log \rho $ in \cite{sz3} due to $\delta-1<0$ when the  density $\rho \rightarrow 0$;
\item
second,  the coefficient $\rho^{\delta-1}$ in front of the Lam\'e operator $ L$ will tend to $\infty$ as $\rho\rightarrow 0$ in the far filed instead of equaling to $1$ in \cite{sz3} or tending to $0$ in \cite{sz333,zz}. Then   it is  necessary  to show that the term $\rho^{\delta-1} Lu$ is well defined.
\end{enumerate}
Recently, via introducing an elaborate (linear) elliptic approach on the operators $L(\rho^{\delta-1}u)$ and some initial compatibility conditions, Xin-Zhu \cite{zz2} identifies one class of initial data admitting one unique 3D local regular solution with far field vacuum and finite energy to the corresponding Cauchy problem of \eqref{eq:1.1} in some inhomogeneous Sobolev spaces. Some other interesting results on the degenerate compressible Navier-Stokes equations   can be found in \cite{chen, cons, ding, zhu, geng22, Germain, guo, Has, JWX,  hailiang,   vassu, Sundbye2, zyj, tyc2}.

In the current paper, we will do some study on the breakdown of the regular solution of  3D degenerate compressible Navier-Stokes equations obtained in \cite{zz2} for the case $\delta \in (0,1)$ (see Theorem 1.1). In order to state our main results clearly, we divide the rest of the introduction into two subsections.

\subsection{Cauchy problem with far field vacuum} Let $\Omega=\mathbb{R}^3$.
We first introduce a proper class of solutions called regular solutions to  the Cauchy problem (\ref{eq:1.1})-(\ref{far}).
\begin{definition}\label{d1}
 Let $T> 0$ be a finite constant. A solution $(\rho,u)$ to the Cauchy problem (\ref{eq:1.1})-(\ref{far}) is called a regular solution in $ [0,T]\times \mathbb{R}^3$ if $(\rho,u)$  satisfies this problem in the sense of distribution and:
\begin{equation*}\begin{split}
&(\textrm{A})\quad \rho>0,\quad   \rho^{\gamma-1}\in C([0,T]; H^3), \quad  \nabla \rho^{\delta-1}\in L^\infty([0,T]; L^\infty\cap D^{2});\\
&(\textrm{B})\quad u\in C([0,T]; H^3)\cap L^2([0,T]; H^4), \quad  u_t \in C([0,T]; H^1)\cap L^2([0,T] ; D^2), \\
&\qquad \ \  \rho^{\frac{\delta-1}{2}}\nabla u \in C([0,T]; L^2), \quad  \rho^{\frac{\delta-1}{2}}\nabla u_t \in  L^\infty([0,T];L^2),\\
&\qquad \ \ \rho^{\delta-1}\nabla u \in L^\infty([0,T]; D^1),\quad   \rho^{\delta-1}\nabla^2 u \in C([0,T]; H^1)\cap   L^2([0,T]; D^2).
\end{split}
\end{equation*}
\end{definition}
\begin{remark}
It follows from  the regularity (\textrm{A}) shown above and the Gagliardo-Nirenberg inequality    that $\nabla \rho^{\delta-1}\in L^\infty$, which means that the vacuum occurs if and only  in the far field. Moreover,  it should be pointed out that the definition of regular solutions above is essentially based on the careful analysis on the   intrinsic singular structures of  \eqref{eq:1.1} for finite energy solutions ($1<\gamma \leq 2$), which can be seen in \S 2 of the current paper, or on pages 8 to 11 of \cite{zz2}. \end{remark}

The local-in-time well-posedenss of the regular solution obtained by Xin-Zhu \cite{zz2} can be given as follows:
\begin{theorem}\cite{zz2}\label{th2} Let  parameters  $(\gamma,\delta, \alpha,\beta)$ satisfy
\begin{equation}\label{canshu}
\gamma>1,\quad 0<\delta<1, \quad \alpha>0, \quad \alpha+\beta\geq 0.
\end{equation}
  If the initial data $( \rho_0, u_0)$ satisfies
\begin{equation}\label{th78}
\begin{split}
&\rho_0>0,\quad (\rho^{\gamma-1}_0, u_0)\in H^3, \quad \nabla \rho^{\delta-1}_0\in D^1\cap D^2,   \quad \nabla \rho^{\frac{\delta-1}{2}}_0\in L^4,\\
\end{split}
\end{equation}
and the initial  compatibility conditions:
\begin{equation}\label{th78zx}
\displaystyle
  \nabla u_0=\rho^{\frac{1-\delta}{2}}_0 g_1,\quad \   \ \  Lu_0= \rho^{1-\delta}_0g_2,\quad \ \
 \nabla \Big(\rho^{\delta-1}_0Lu_0\Big)=\rho^{\frac{1-\delta}{2}}_0g_3,
\end{equation}
for some $(g_1,g_2,g_3)\in L^2$, then there exist a  time $T_*>0$ and a unique regular solution $(\rho, u)$  in $ [0,T_*]\times \mathbb{R}^3$ to the Cauchy problem (\ref{eq:1.1})-(\ref{far}) satisfying:
\begin{equation}\label{reg11}\begin{split}
& t^{\frac{1}{2}}u\in L^\infty([0,T_*];D^4),\quad  t^{\frac{1}{2}}u_t\in L^\infty([0,T_*];D^2)\cap L^2([0,T_*] ; D^3),\\
&u_{tt}\in L^2([0,T_*];L^2),\quad   t^{\frac{1}{2}}u_{tt}\in L^\infty([0,T_*];L^2)\cap L^2([0,T_*];D^1),\\
& \rho^{1-\delta}\in L^\infty([0,T_*];  L^\infty\cap D^{1,6} \cap D^{2,3} \cap D^3), \\
&  \nabla \rho^{\delta-1}\in C([0,T_*]; D^1\cap D^2), \  \nabla \log \rho \in L^\infty([0,T_*];L^\infty\cap L^6\cap  D^{1,3} \cap D^2).
\end{split}
\end{equation}

Moreover, if $1<\gamma\leq 2$, $(\rho, u)$ is a classical solution to   (\ref{eq:1.1})-(\ref{far}) in $(0,T_*]\times \mathbb{R}^3$.
\end{theorem}
\begin{remark}
For $(\alpha,\beta)$, it is  required that $\alpha>0$ and $2\alpha+3\beta\geq 0$ in \cite{zz2}, which, via the same arguments used in their paper, can be easily replaced by  \eqref{10000}. \end{remark}

Naturally we will   consider  that the local  regular solutions obtained above  to  the Cauchy problem (\ref{eq:1.1})-(\ref{far}) may cease to exist globally,  or what is the key estimate to make sure that this solution  could be extended to be   a  global one?
The similar question has been studied for the 3D incompressible Euler equation  by Beale-Kato-Majda (BKM) in their pioneering work \cite{TBK},  in which they showed:  if  $0 < \overline{T} < +\infty$  is the
maximum existence time for the smooth solution, then the $L^\infty$-bound of vorticity  must blow up, i.e., 
\begin{equation}\label{kaka1}
\lim \sup_{ T \rightarrow \overline{T}} \int_0^T \|\nabla \times u(t,\cdot)\|_ {L^\infty(\mathbb{R}^3)}\text{d}t=\infty.
\end{equation}
Later, Ponce  \cite{pc} rephrased the above criterion in terms of the deformation tensor $D(u)$, which has been applied to the strong solution with vacuum (see \cite{CK3}) for   some 3D compressible  viscous flow systems \cite{hup,szz} for the case $\delta=0$ in \eqref{fandan}.

 One of our main results  in the following   theorem shows  that, for the  degenerate compressible viscous flow with $0<\delta <1$ in \eqref{fandan},  the $L^\infty$ norm of the deformation tensor $D(u)$ and $L^6$ (or $L^\infty$) norm of $\nabla \rho^{\delta-1}$ control the possible breakdown of  the unique regular solution with far field vacuum in the whole space, which  means that if a solution of the \textbf{ICND} system  is initially regular and loses its regularity at some later time, then the formation of singularity must be caused  by  losing  the bound of  $D(u)$ or $\nabla \rho^{\delta-1}$ as the critical time approaches; equivalently, if both $D(u)$ and   $\nabla \rho^{\delta-1}$ remain bounded, a regular  solution persists. This conclusion can be stated precisely  as follows.

\begin{theorem}\label{th3s}
Assume that  \eqref{canshu} holds.   Let $(\rho(t,x), u(t,x))$ be the unique  regular solution obtained in Theorem \ref{th2} to the Cauchy problem (\ref{eq:1.1})-(\ref{far}).
If  $\overline{T}< +\infty $ is its maximal existence time, then  
\begin{equation}\label{cri}
\begin{split}
\displaystyle\lim_{T\mapsto \overline{T}} \left(\sup_{0\le t\le T}\big\|\nabla \rho^{\delta-1}(t,\cdot)\big\|_{L^6(\mathbb{R}^3)}+\int_0^T \|D( u)(t,\cdot)\|_ {L^\infty(\mathbb{R}^3)}\ \text{\rm d}t\right)=\infty,
\end{split}
\end{equation}
\begin{equation}\label{criinfty}
\begin{split}
\displaystyle\lim_{T\mapsto \overline{T}} \left(\sup_{0\le t\le T}\big\|\nabla \rho^{\delta-1}(t,\cdot)\big\|_{L^\infty(\mathbb{R}^3)}+\int_0^T \|D( u)(t,\cdot)\|_ {L^\infty(\mathbb{R}^3)}\ \text{\rm d}t\right)=\infty.
\end{split}
\end{equation}

\end{theorem}

\begin{remark}\label{classicalsense}
Actually, if we assume that for any time $T^*>0$,
$$
\displaystyle\lim_{T\mapsto T^*} \left(\sup_{0\le t\le T}\big\|\nabla \rho^{\delta-1}(t,\cdot)\big\|_{L^6(\mathbb{R}^3)}+\int_0^T \|D( u)(t,\cdot)\|_ {L^\infty(\mathbb{R}^3)}\ \text{\rm d}t\right)<\infty,
$$
or 
$$
\displaystyle\lim_{T\mapsto T^*} \left(\sup_{0\le t\le T}\big\|\nabla \rho^{\delta-1}(t,\cdot)\big\|_{L^\infty(\mathbb{R}^3)}+\int_0^T \|D( u)(t,\cdot)\|_ {L^\infty(\mathbb{R}^3)}\ \text{\rm d}t\right)<\infty,
$$
then for $\gamma \in (1,2]$,  the regular solution  $(\rho,u)$ obtained in Theorem 1.1   satisfies the Cauchy problem (\ref{eq:1.1})-(\ref{far}) in the classical sense  in $(0,T^*]\times \mathbb{R}^3$.

Second, when $\gamma \in (1,2]$, the solution obtained in Theorem 1.1  has finite energy.

\end{remark}
\subsection{Periodic problem away from the vacuum}

Let $\Omega =\mathbb{T}^3$. Now we state that,  under the additional assumption that the shear and second viscosities (respectively $\mu(\rho)$ and $\lambda(\rho)$) satisfy the B-D relation \eqref{bdrelation}(see \cite{bd6,bd2,bd7,bd}), if we consider the corresponding problem in some periodic domain and the initial density is away from the vacuum, the possible breakdown of classical solutions can be controlled only by the $L^\infty$ norm of $D(u)$. This conclusion can be stated precisely as follows.

\begin{theorem}\label{th3t}
Let   \eqref{canshu} hold, and 
\begin{equation}\label{bdrelation}
\lambda=2(\mu'(\rho)\rho-\mu(\rho))=2\alpha (\delta-1)\rho^\delta.
\end{equation}
Assume the initial data $(\rho_0,u_0)$ satisfies 
\begin{equation}\label{bdchuzhi}
\rho_0>0,\quad (\rho_0,u_0) \in H^3,
\end{equation}
and $(\rho(t,x), u(t,x))$ is the corresponding   unique  classical solution in $[0,T]\times \mathbb{T}^3$  for some positive time $T>0$ to the periodic problem  (\ref{eq:1.1})-(\ref{initial}) which satisifes
$$
\rho\in C([0,T]; H^3),  \quad u\in C([0,T]; H^3)\cap L^2([0,T];H^4).
$$
If the  maximal existence time  $\overline{T}$ of this solution is  finite, then  
\begin{equation}\label{crit}
\begin{split}
\displaystyle\lim_{T\mapsto \overline{T}} \int_0^T \|D( u)(t,\cdot)\|^2_ {L^\infty(\mathbb{T}^3)}\ \text{\rm d}t=\infty.
\end{split}
\end{equation}

\end{theorem}
\begin{remark}
Actually, the conclusion obtained above can also be applied to some initial boundary value problems of the system (\ref{eq:1.1}) with $\delta \in (0,1)$ in \eqref{fandan} under proper boundary conditions in smooth and bounded domains. For simplicity, here we only consider the periodic problem.
\end{remark}

The rest of this paper will be divided into 5 sections. In \S 2, we introduce two intrinsic singular structures \eqref{dege} and \eqref{we22}  of the system \eqref{eq:1.1}, and then show the main strategy of our proof. In \S 3, we show some new elliptic approaches that are related to singular structures \eqref{dege} and \eqref{we22}. Moreover, in order to make sure that we can continue to apply the local existence theory obtained in Theorem 1.1 at any positive time within the solution’s life span, we also verify all the compatibility conditions and initial conditions for our solution at any positive time before the singularity appears. Based on the analysis on the mathematical structure of system \eqref{eq:1.1}, we give the proof for Theorem 1.2 in \S 4. \S 5 is devoted to proving Theorem 1.3. Finally, we will give an appendix to list some basic lemmas that were used frequently in our proof. It is worth pointing out that, except the conclusions mentioned above, another purpose of the current paper is to show how to understand the intrinsic singular structures of the fluid system considered now, and then how to develop the corresponding nonlinear energy estimates for the regular solution with finite energy.

\section{Reformulation and main strategy}

In this section,    we always assume that \eqref{canshu} holds and $\Omega=\mathbb{R}^3$. We first introduce two  intrinsic singular structures of system  \eqref{eq:1.1}, and then show the main strategy of our proof.  For simplicity, in the rest of this paper, we always  denote 
 \begin{equation} \label{xishu}
\begin{split}
& a=\Big(\frac{A\gamma}{\gamma-1}\Big)^{\frac{1-\delta}{\gamma-1}},\quad \text{and} \quad e=\frac{\delta-1}{2(\gamma-1)}<0.
\end{split}
\end{equation}
Moreover, we use $\textbf{DTE}$ to denote the  degenerate time  evolution operator, $\textbf{WSS}$ to denote the  weak singular  source term, $\textbf{SSS}$ to denote the  strong  singular  source term, and $\textbf{SSE}$ to denote the strong singular elliptic operator. 

Generally, because the  momentum equations $(\ref{eq:1.1})_2$ is a double degenerate system when the density loses its positive lower bound, i.e., 
$$
\displaystyle
 \underbrace{\rho(u_t+u\cdot \nabla u)}_{\textbf{DTE}}+\nabla P= \underbrace{\text{div}(2\mu(\rho)D(u)+\lambda(\rho)\text{div}u \mathbb{I}_d)}_{Degenerate\ elliptic \ operator},
$$
usually,  it is very hard to obtain higher order regularities  of the fluids velocity $u$ near the vacuum. Then it is necessary that we need to  find some intrinsic   structures  of this system  to make  some effective analysis on $u$. Due to $\delta \in (0,1)$, formally,  we have two choices. The first one  is  the``Degenerate"--``Weak-Singular"   structure shown in  \eqref{dege},  which  has a degeneracy in the time evolution, but provides one uniform elliptic operator $Lu$.  However, this structure still has one $\textbf{WSS}$: $\nabla \log \rho\cdot Q(u)$. The other one is the 
strong singular  structure shown in  \eqref{we22}, which has a nice time evolution operator, but also has one $\textbf{SSE}$: $\rho^{\delta-1}Lu$ and one $\textbf{SSS}$: $\nabla   \rho^{\delta-1} \cdot Q(u)$.

\subsection{``Degenerate"--``Weak-Singular"   structure}
In terms of   variables 
$$
\varphi=a \rho^{1-\delta},\quad g=\frac{aA\gamma}{\gamma-\delta}\rho^{\gamma-\delta},\quad f=a\delta \nabla \log \rho=(f^{(1)},f^{(2)},f^{(3)}), $$
and $u$, the system (\ref{eq:1.1})  can be rewritten as
\begin{equation}
\begin{cases}
\label{dege}
\displaystyle
\varphi_t+u\cdot \nabla \varphi+(1-\delta)\varphi \text{div} u=0,\\[12pt]
\displaystyle
g_t+u\cdot \nabla g+(\gamma-\delta)g\text{div} u=0,\\[12pt]
\displaystyle
 \underbrace{\varphi(u_t+u\cdot \nabla u)}_{\textbf{DTE}} +\nabla g+aLu= \underbrace{f \cdot Q(u)}_{\textbf{WSS}},\\[12pt]
\displaystyle
f_t+\sum_{l=1}^3 A_l \partial_lf+B^*f+a\delta\nabla \text{div} u=0,
 \end{cases}
\end{equation}
 where 
$A_l=(a^l_{ij})_{3\times 3}$ for  $i,j,l=1,2,3$,
are symmetric  with $a^l_{ij}=u^{(l)}\quad \text{for}\ i=j$; otherwise $a^l_{ij}=0$,   and  $B^*=(\nabla u)^\top$.

\subsection{Strong singular  structure}

In terms of   variables
\begin{equation}\label{bianliang}\phi=\frac{A\gamma}{\gamma-1} \rho^{\gamma-1},\quad \psi=\frac{\delta}{\delta-1}\nabla \rho^{\delta-1}=(\psi^{(1)},\psi^{(2)},\psi^{(3)}),
\end{equation}
and $u$,  the system (\ref{eq:1.1})  can be rewritten as
\begin{equation}
\begin{cases}
\label{we22}
\displaystyle
\phi_t+u\cdot \nabla \phi+(\gamma-1)\phi \text{div} u=0,\\[12pt]
\displaystyle
u_t+u\cdot\nabla u +\nabla \phi+ \underbrace{a\phi^{2e}Lu}_{\textbf{SSE}}
= \underbrace{\psi \cdot Q(u)}_{\textbf{SSS}},\\[12pt]
\displaystyle
\psi_t+\sum_{l=1}^3 A_l \partial_l\psi+B\psi+ \underbrace{\delta a\phi^{2e}\nabla \text{div} u}_{\textbf{SSS}}=0,
 \end{cases}
\end{equation}
where $B=(\nabla u)^\top+(\delta-1)\text{div}u\mathbb{I}_3$. 
In Sections 3-4, we will consider the Cauchy problem \eqref{we22} with the   initial data
\begin{equation} \label{sfana1}
\begin{split}
&(\phi, u,\psi)|_{t=0}=(\phi_0, u_0, \psi_0)\\
=&\Big(\frac{A\gamma}{\gamma-1} \rho^{\gamma-1}_0(x),   u_0(x), \frac{\delta}{\delta-1}\nabla \rho^{\delta-1}_0(x)\Big),\quad x\in \mathbb{R}^3,
\end{split}
\end{equation}
and  the far field behavior:
\begin{equation}\label{sfanb1}
\begin{split}
(\phi, u,\psi)\rightarrow (0,0,0),\quad \text{as}\quad  |x|\rightarrow +\infty,\quad t \geq 0.
\end{split}
\end{equation}

For the simplicity of the proof in Sections 3-4, we also give some relations between the new variables:
$$
\varphi=\phi^{-2e},\quad f=\psi \varphi, \quad g=\frac{\gamma-1}{\gamma-\delta}\phi\varphi, \quad \psi=\frac{a\delta}{\delta-1}\nabla \phi^{2e}=\frac{a\delta}{\delta-1}\nabla \varphi^{-1}.$$

\subsection{Main strategy}
Now, based on the two  intrinsic  singular   structures  of the system \eqref{eq:1.1} mentioned above, now we show our main strategy of our proof for Theorem \ref{th3s}.
\subsubsection{Necessity of  the strong singular  structure \eqref{we22}}

It is well known that  the single degenerate   structure
\begin{equation}\label{sd}
 \underbrace{\rho(u_t+u\cdot \nabla u)}_{\textbf{DTE}} +\nabla P+Lu=0,
\end{equation}
 has been widely used for the well-posedess or singularity formation  theory of smooth solutions with vacuum to  compressible fluid systems  for the case $\delta=0$ (\cite{CK3, hup,szz}). For such kind of structures,  we can make sure that the velocity belongs to $D^1\cap D^3$ and also $\sqrt{\rho}u\in L^2$,    if the initial data is smooth enough and satisfy some necessary compatibility conditions, which implies that the velocity itself only belongs $L^6\cap L^\infty$ in the whole space.

However, for our ``Degenerate"--``Weak-singular"   structure \eqref{dege},  if   the velocity  $u$ only belongs to $D^1\cap D^3$ and  $\sqrt{\varphi}u\in L^2$, it  is not good enough to close the desired energy estimates.   For example,  considering  the basic weighted $L^2$ norm of $u$: $|\sqrt{\varphi}u|_2$, it follows from the equations $\eqref{dege}_3$ that 
 \begin{equation}\label{hardpoint}
\begin{split}
&\frac{1}{2}\frac{d}{dt}\int \varphi u^2+a\int (\alpha |\nabla u|^2+(\alpha+\beta)|\text{div}u|^2)\\
=& \frac{1}{2}\int \varphi_t u^2-\int \varphi (u\cdot \nabla)u\cdot u-\int \nabla g \cdot u+ \underbrace{\int f\cdot Q(u)\cdot u}_{\textbf{Trouble term $I^*$}} .
\end{split}
\end{equation} 
 For the last term 
 $$
 I^*=\int f\cdot Q(u)\cdot u, $$
 one has $Q(u)\in L^2\cap L^\infty$ and $u\in L^6\cap L^\infty$. Then in order to make sure that 
 $$f\cdot Q(u)\cdot u \in L^1,$$ at least we need   $f\in L^p$ for   some $p\in [1,3]$. However,  even $\rho>0$ only decays to zero in the far field. It is  still very hard to find the initial data such that 
 $$
 f(0,x)=a\delta \nabla \log \rho_0 \in L^p \quad \text{for} \quad \text{some} \quad p\in [1,3].
 $$  
 Generally, for some initial density decays to zero in the form of polynomials:
 $$
\rho_0(x)=\frac{1}{1+|x|^{2q}} \quad \text{for \ some \ proper} \quad q,
$$
it is easy to see that, actually, 
 $$
 f(0,x)=a\delta \nabla \log \rho_0 \in L^p \quad \text{for} \quad \text{any} \quad p>3.
 $$  
 We know that $f$ should still keep in $L^p$ within the life span due to the standard theory of the symmetric hyperbolic system $\eqref{dege}_4$.
 Then according to the H\"older's inequality, in order to make sure that $f\cdot Q(u)\cdot u \in L^1$, at least we need
 $$
 u\in L^l, 
 $$
 where $l=\frac{2p}{p-2} \in [2,6)$. Formally we know that the ``Degenerate"--``Weak-Singular"  structure \eqref{dege} at most provides the information $u\in L^6\cap L^\infty$ due to the degeneracy in the time evolution.  Then it is obvious that the ``Degenerate”–``Weak-Singular” structure (\ref{dege}) can not give enough information for closing the desired nonlinear energy estimates.   Then we need to introduce another structure  \eqref{we22} to give the $L^2$ integrability of  $u$.

Next we formally  show how to use the system \eqref{we22} to give one close energy estimates for the nonlinear problems. First,   for   the behavior of the possible singular term $\nabla \rho^{\delta-1}$, it follows from $\eqref{we22}_3$ that it   could be controlled by a     symmetric hyperbolic  system with a possible singular higher order term $\delta a\phi^{2e}\nabla \text{div} u$.   Due to the fact that $\phi^{2e}$  has an  uniformly positive lower bound in the whole space, then
for this  special strong singular system, one can find formally that, even though the coefficients $a\phi^{2e}$ in front  of  Lam\'e operator $ L$ will tend to $\infty$ as $\rho\rightarrow 0$ in the far filed, yet this structure could give a better a priori estimate on $u$ in weighted-$H^3$ than those of \cite{CK3,sz3,sz333,zz}.  In order to close the estimates, we need to control $\phi^{2e}\nabla \text{div} u$ in $D^1\cap D^2$, which  can be obtained by regarding the momentum equations as the following inhomogeneous Lam\'e  equations:
$$
a L(\phi^{2e}\partial_k u)=a \phi^{2e}L\partial_ku-aG(\nabla \phi^{2e},\partial_ku)$$
with
$$
G(v,u)= \alpha  v\cdot \nabla u+\alpha \text{div}(u\otimes v)+(\alpha+\beta)\big(v\text{div}u+v \cdot (\nabla u)+u  \cdot \nabla v\big).
$$
Actually,  one has
\begin{equation}\label{singularelliptic}
\begin{split}
|\phi^{2e}\nabla^2 u|_{D^1}\leq & C(|\phi^{2e}\nabla u|_{D^2}+|\psi|_\infty|\nabla^2u|_2+|\nabla u|_\infty|\nabla \psi|_2).
\end{split}
\end{equation}
 Similar calculations can be done for  $|\phi^{2e}\nabla^2u|_{D^2}$ and $|\phi^{e}\nabla^2u|_{2}$.  It should be pointed out that we have used the following facts to satisfy the requirement of the far field behavior in the standard elliptic regularity theory (see Lemma \ref{zhenok}):
 $$
 \phi^{2e}\nabla u \in L^\infty([0,T];L^6),\quad  \phi^{2e}\nabla^2 u \in L^\infty([0,T];L^2),\quad   \phi^{e}\nabla u \in L^\infty([0,T];L^2).$$

 Thus, it seems that at least, the reformulated system (\ref{we22}) can provide a closed energy  estimates for the  regular solution we defined in the introduction.

 \subsubsection{Advantage of the  ``Degenerate"--``Weak-Singular"   structure \eqref{dege} }
From \eqref{we22}, the assumption on the boundedness of $\psi$ and $D(u)$ in \eqref{we11} can only ensure the boundedness of $\nabla^2 u$ and $\phi^{2e}Lu$. For the higher order estimates on   $\nabla^k u$ $(k = 3,4)$ with singular weights, we need the estimates on $\nabla^l \psi$  $(l = 1, 2)$. Actually, the thing will become very complex if we want to do these estimates directly from \eqref{we22}. The introduction of \eqref{dege} will make the estimates look clearer. 

Taking the estimate on $\|u\|_{L^\infty([0,T];D^3)}$ for example, first we need to consider 
\begin{equation}\label{zhucc}
\begin{split}
& \frac{1}{2}\frac{d}{dt}\Big(a\alpha|\phi^e\nabla u_t|^2_2+a(\alpha+\beta)|\phi^e\text{div} u_t|^2_2\Big)+| u_{tt}|^2_2\\
=&\int \Big(-(u\cdot \nabla u)_t-\nabla \phi_t-a \phi^{2e}_t Lu -a \nabla \phi^{2e} \cdot Q(u_t) \Big)\cdot u_{tt} \\
&+\int \Big(\frac{a}{2} \phi^{2e}_t\big(\alpha|\nabla u_t|^2+(\alpha+\beta)|\text{div}u_t|^2\big)+\underbrace{(\psi \cdot Q(u))_t\cdot u_{tt}}_{\textbf{Trouble term $J^*$}}  \Big).
\end{split}
\end{equation}
For the one component of the  last term $J^*$
$$
\int \psi \cdot Q(u)_t \cdot u_{tt},
$$
it can only be controlled by  $|\psi \cdot Q(u)_t |_2|u_{tt}|_2$. Here, we may  use $|\psi|_\infty|Q(u)_t |_2$ or $|\psi|_6| Q(u)_t |_3$ to control the  norm $|\psi \cdot Q(u)_t |_2$. However, in order to control $|\psi|_\infty$, we need $\|\psi\|_{D^1\cap D^2}$, and then $|\phi^{2e}\nabla^2 u|_{L^1([0,T]; D^1\cap D^2)}$. This  means  that  at least we need the estimate on $\int_0^t |\nabla^2 u_t|^2_2\text{d}s$, which is also required in  $| Q(u)_t |_3$.

 If we continue to use  the  structure (\ref{we22}) for obtaining the estimate on $\int_0^t |\nabla^2 u_t|^2_2 \text{d}s$, formally one has 
$$
a L(\phi^{2e}u_t)=-\phi^{2e}\big(\phi^{-2e}(u_t+u\cdot\nabla u +\nabla \phi-\psi\cdot Q(u))\big)_t -\frac{\delta-1}{\delta}G(\psi,u_t)=Y.
$$
On the one hand, the  following only way to estimate $|\nabla^2 u_t|_2$    might not hold:
\begin{equation}\label{possibleguji}
|\phi^{2e}u_t|_{D^2}\leq C|Y|_2,
\end{equation}
for some constant $C>0$. Because we do not have $\phi^{2e}u_t \in L^p$ for some $p>0$, or in some weak or strong sense,
$$\phi^{2e}u_t\rightarrow 0\quad \text{as} \quad |x|\rightarrow +\infty.$$
The similar thing also happens for the estimate for $|\phi^{2e}\nabla^2 u|_2$. On the other hand, the estimate on $|Y|_2$ at least depends on $|\nabla \psi |_2 + |\nabla^2 \phi|_2 + |\phi^{2e}\nabla^2 u|_2$, which, from (2.9), (4.56) and Lemma 3.3, is controlled by $|\nabla f |_2 + \|\nabla^2 \phi\|_1$. Actually, at the current step, what we need is only the estimate on $|\nabla^2 u_t|_2$, not on $|\phi^{2e}\nabla^2 u_t|_2$. Then we should ask help from the structure \eqref{dege}, where one has
 \begin{equation}\label{elss}
aLu_t=-\varphi u_{tt} -\varphi (u\cdot\nabla u)_t -\varphi_t (u_t+u\cdot\nabla u)-\nabla g_t+(f\cdot Q(u))_t.
\end{equation}
It follows from  Lemma \ref{zhenok} that 
 \begin{equation}\label{zhss}
\begin{split}
|u_t|_{D^2}\leq& C(|-\varphi u_{tt} -\varphi (u\cdot\nabla u)_t -\varphi_t (u_t+u\cdot\nabla u)-\nabla g_t+(f\cdot Q(u))_t|_2).
\end{split}
\end{equation}
for some constant $C>0$, which only depends on $|\nabla f|_2$  but not on $|\nabla \psi|_2$.  The estimate on $|\nabla f|_2$ is much easier than that of $|\nabla \psi|_2$, which 
makes the  estimates  clearer compared with the  one in (\ref{possibleguji}). 
Once  obtaining  the estimate on $\int_0^t |\nabla^2 u_t|^2_2\text{d}s$, then we can come back to the   structure (\ref{we22}) again for further estimates, which is good enough for extending our solution beyond the time $\overline{T}$ (see \S 4.2).

\section{Elliptic argument and compatibility conditions}

In this section,    we always assume that \eqref{canshu} holds and $\Omega=\mathbb{R}^3$. 
First, we will show some elliptic approaches that are related to singular structures  \eqref{dege} and \eqref{we22}. Second, in order to make sure that we can continue to apply the local existence theory obtained in Theorem 1.1 at any positive time, we will verify the compatibility conditions in (1.14) for any positive time within the regular solution’s life span. At last, we will show some additional regularities of the regular solution that are not shown in Thoerem 1.1.

\subsection{Standard  elliptic approach}

The following regularity estimate for the Lam$\acute{ \text{e}}$ operator is standard in harmonic analysis.
\begin{lemma}\cite{CK3, harmo}\label{zhenok} Let $1< q< +\infty$, $k\in \mathbb{Z}$ and $Z\in D^{k,q}$. If  $u\in D^{1,q}_0(\mathbb{R}^3)$   is a weak solution to the following elliptic  problem
\begin{equation}
\label{ok}
\displaystyle
Lu=Z, 
\end{equation}
then $u\in D^{k+2,q}$, and  it holds that
$$
|u|_{D^{k+2,q}} \leq C |Z|_{D^{k,q}},
$$
where the constant $C>0$ depends only on $\alpha$,  $\beta$  and $q$.
\end{lemma}

\subsection{Density involved elliptic approach}  Let $T>0$ be some time. In the rest of this section, let $(\rho,u)$ in $[0,T]\times \mathbb{R}^3$ be the unique regular solution to the Cauchy problem \eqref{eq:1.1}-\eqref{far} obtained in Theorem 1.1. For simplicity, we denote
\begin{equation*}
\begin{split}
G(v,u)=& \alpha  v\cdot \nabla u+\alpha \text{div}(u\otimes v)+(\alpha+\beta)\big(v\text{div}u+v \cdot (\nabla u)+u  \cdot \nabla v\big),\\
H(u,\phi,\psi)=&H=u_t+u\cdot\nabla u +\nabla \phi-\psi\cdot Q(u).
\end{split}
\end{equation*}

\begin{lemma}\label{5jie}For $k=1,2,3$,
\begin{equation*}
\begin{split}
|\phi^{e}\partial_k u|_{D^2}\leq & C(|\phi|^{-e}_\infty|\partial_kH|_2+|\partial_k \phi^{-e}H|_2+|G( \phi^{-e}\psi,\partial_ku)|_2),
\end{split}
\end{equation*}
where the constant $C>0$ depends only on $\alpha$,  $\beta$, $A$, $\gamma$ and $\delta$.

\end{lemma}
\begin{proof}According to  the  equations $ (\ref{we22})_2$,  for $k=1,2,3$,      one has
\begin{equation}\label{elliptic-yijie1}
\begin{split}
a L(\phi^{e}\partial_k u)=&a \phi^{e}L\partial_ku-a G(\nabla \phi^e,\partial_ku)\\
=&-\phi^{e}\partial_k\big(\phi^{-2e}H\big) -\frac{\delta-1}{2 \delta}G( \phi^{-e}\psi,\partial_ku)\\
=&- \phi^{-e}\partial_kH-2\partial_k \phi^{-e}H -\frac{\delta-1}{2 \delta}G( \phi^{-e}\psi,\partial_ku).
\end{split}
\end{equation}
\end{proof}

\begin{lemma}\label{3jie}For $k=1,2,3$, 
$$
|\phi^{2e}\partial_k u|_{D^2}\leq C(|H|_{D^1}+|fH|_2+|G(\psi,\partial_k u)|_2),
$$
where the constant $C>0$ depends only on $\alpha$,  $\beta$, $A$, $\gamma$ and $\delta$.

\end{lemma}
\begin{proof}According to  the  equations $ (\ref{we22})_2$, for $k=1,2,3$, one has
\begin{equation}\label{elliptic-syijie1}
\begin{split}
a L(\phi^{2e}\partial_k u)=&a \phi^{2e}L\partial_ku-aG(\nabla \phi^{2e},\partial_ku)\\
=&-\phi^{2e}\partial_k\big(\phi^{-2e}H\big) 
-\frac{\delta-1}{\delta}G(\psi,\partial_k u)\\
=&-\partial_k H-\frac{1-\delta}{a\delta} f H-\frac{\delta-1}{\delta}G(\psi,\partial_k u).
\end{split}
\end{equation}
\end{proof}

\begin{lemma}\label{4jie}  
$$
|\phi^{2e}\nabla^\xi u|_{D^2}\leq C(|\nabla^\xi H|_2+|f \cdot \nabla H|_2+||f|^2H|_2+|\nabla f\cdot H|_2+|G(\nabla \phi^{2e},\nabla^\xi u )|_2),
$$
where the constant $C>0$ depends only on $(\alpha, \beta, A, \gamma, \delta)$, and $\xi\in \mathbb{R}^3$ with $|\xi|=2$ is any multi-index whose components are all non-negative integers.
\end{lemma}
\begin{proof}
According to  the equations $ (\ref{we22})_2$, for multi-index $\xi\in \mathbb{R}^3$, one has
\begin{equation}\label{jiegou1}
\begin{split}
a L(\phi^{2e} \nabla^\xi u)=& a \phi^{2e} \nabla^\xi Lu-aG(\nabla \phi^{2e},\nabla^\xi u)\\
=&-\phi^{2e} \nabla^\xi \big( \phi^{-2e} H \big)-\frac{\delta-1}{\delta}G(\psi,\nabla^\xi u ).
\end{split}
\end{equation}

\end{proof}
\begin{remark}
In the above three lemmas, one has used the facts: for $k=1$, $2$,
$$
\phi^e\nabla u\in H^1(\mathbb{R}^3),\quad \phi^{2e}\nabla^{k} u\in L^6(\mathbb{R}^3).$$

\end{remark}
\subsection{Verification of compatibility conditions}

According to Theorem 1.1 and the initial assumption \eqref{th78}-\eqref{th78zx}, now we have the unique regular solution $(\rho, u)$ to the Cauchy problem \eqref{eq:1.1}-\eqref{far} in $[0,T]\times \mathbb{R}^3$. In order to extend this local solution from $[0, T ]$  to the time interval $[0, T_1]$ with $T_1 > T$,  we need to apply Theorem 1.1 at $t = T$ again. Thus we need to make sure that the compatibility condition \eqref{th78zx} still holds at time $t = T$. For this purpose, we first need the following lemma:
\begin{lemma}
\begin{equation}\label{verificationcom}
\begin{split}
&\phi^e\nabla u\in C([0,T]; L^2),\quad \phi^{2e} Lu\in C([0,T]; L^2),\\
&  \phi^e\nabla u_t\in L^\infty([0,T]; L^2),\quad \phi^{2e}\nabla^2 u_t\in L^2([0,T]; L^2),\\
& t^{\frac{1}{2}}\phi^e\nabla u_{tt}\in L^2([0,T]; L^2),  \quad  t\phi^{e}\nabla (\phi^{2e} Lu)\in C([0,T]; L^2).
\end{split}
\end{equation}
\end{lemma}
\begin{proof}
First, the first line  and the first property in the second line of  \eqref{verificationcom} can be obtained directly from Theorem 1.1.  

Second, the second property  in the second line and the first one in the third line  of  \eqref{verificationcom}, actually have been proven in \cite{zz2}, but is not indicated clearly in Theorem 1.1 and Definition 1.1, which can be implied by the
estimates (3.74) and (3.80), convergences (3.85), (3.127) and (3.141) of \cite{zz2}. 

Next, for the second   property in the third  line of  \eqref{verificationcom}, it follows from equations $\eqref{we22}_2$, Definition 1.1,   the proved properties of \eqref{verificationcom}  and Theorem 1.1 that
\begin{equation}\label{verificationcomuy}
\begin{split}
a\phi^{3e}\partial_k Lu=&-a\partial_k \phi^{3e}Lu-\partial_k\phi^e u_t-\phi^e\partial_k u_t-\phi^e\partial_k(u\cdot \nabla u)\\
&-\partial_k \phi^e u\cdot \nabla u-\partial_k(\phi^e\nabla \phi)+\partial_k\phi^e\psi\cdot Q(u)\\
&+\phi^e \partial_k(\psi\cdot Q(u)) \in L^\infty([0,T]; L^2),
\end{split}
\end{equation}
for $k = 1$, $2$, $3$.

It follows from the following relations
\begin{equation}
\begin{cases}
\label{we22uy}
\displaystyle
\phi^e_t=-u\cdot \nabla \phi^e-\frac{\delta-1}{2}\phi^e \text{div} u,\\[12pt]
\displaystyle
\psi_t=-\sum_{l=1}^3 A_l \partial_l\psi-B\psi-\delta a\phi^{2e}\nabla \text{div} u,
 \end{cases}
\end{equation}
Definition 1.1 and Theorem 1.1 that
\begin{equation}
\label{xinxiuy}
\displaystyle
\phi^e_t, \quad  \nabla \phi^e_t, \quad \psi_t, \quad \nabla \psi_t \in L^\infty([0,T];L^2),
\end{equation}
which, along with the relation \eqref{verificationcomuy}, implies that
\begin{equation}
\label{xinxi2uy}
\displaystyle
at(\phi^{3e}\partial_k Lu)_t\in L^2([0,T];L^2).
\end{equation}
Then it follows from  \eqref{verificationcomuy}, \eqref{xinxi2uy} and the classical Sobolev imbedding theorem that
\begin{equation}
\label{xinxi3uy}
\displaystyle
a t \phi^{3e}\partial_k Lu\in C([0,T];L^2).
\end{equation}

Similarly, one  can also show that
\begin{equation}
\label{xinxi4uy}
\displaystyle
a\phi^{e}\partial_k\phi^{2e} Lu\in C([0,T];L^2),
\end{equation}
 which, together with \eqref{xinxi3uy}, implies the desired conclusion.

The above lemma implies that, if  $(\rho,u)$ is the unique regular solution in $[0,T]\times \mathbb{R}^3$ to the Cauchy problem \eqref{eq:1.1}-\eqref{far}, then one can obtain that
\begin{equation}\label{specal0}
(\phi^e \nabla u)(T,x), \quad (\phi^{2e} Lu)(T,x), \quad (\phi^{e} \nabla (\phi^{2e} Lu))(T,x) \in L^2(\mathbb{R}^3).
\end{equation}
This means that the compatibility conditions in \eqref{th78zx} are still availible at $t = T$.
\end{proof}

\subsection{Verification of some special initial condition}
In the initial assumption \eqref{th78} of Theorem 1.1, the authors require that
\begin{equation}
\label{special1}
\displaystyle
\nabla \rho^{\frac{\delta-1}{2}}_0 \in L^4,
\end{equation}
which, actually, is only used in the initial data’s approximation process from the non-vacuum flow to the flow with the far field vacuum, and does not appear in the corresponding energy estimates process in the proof for the local existence of the
regular solution shown in \cite{zz2}. Therefore, in Theorem 1.1, the regularity of the quantity $\nabla \rho^{\frac{\delta-1}{2}}$ has not been mentioned.

However, whether the fact  $\nabla \rho^{\frac{\delta-1}{2}} \in C([0,T];L^4)$ holds or not   is very crucial for our following proof, not only for the application of the local existence theorem at some
positive time, but also for the estimate on $\phi^e \nabla u_t$ (see Lemma 4.11). Now we will verify this desired information in this subsection.

Denote
$$
\omega=\nabla \rho^{\frac{\delta-1}{2}},
$$
 it follows from the equation $\eqref{eq:1.1}_1$ that $\omega$ satisfies the following equations
\begin{equation}
\label{special2}
\displaystyle
\omega_t+\sum_{l=1}^3 A_l \partial_l\omega+A^*(u)\omega+\frac{\delta-1}{2}\sqrt{a} \phi^{e}\nabla \text{div} u=0,
\end{equation}
where $A^*(u)=(\nabla u)^\top+\frac{\delta-1}{2}\text{div}u\mathbb{I}_3$.

\begin{lemma}\label{special3}
$$
 \omega\in C([0,T];L^6\cap L^\infty\cap D^{1,3}\cap D^2), \quad \omega_t \in L^\infty([0,T];H^1).
$$
\end{lemma}
The proof of the above lemma can be easily got from the regularity of the solution
$(\rho,u)$ obtained in Theorem 1.1. Here we omit it. 

\begin{lemma}\label{special4}
$$  \omega \in C([0,T];L^4)\quad \text{and}\quad  \nabla f\in L^\infty([0,T]; L^2).$$
\end{lemma}
\begin{proof} According to Lemma \ref{special3} and the Sobolev imbedding theorem, for the time continuity of $\omega$, we only need
to show that
\begin{equation}
\label{special5}
 \omega \in L^\infty([0,T];L^4).
\end{equation}

In fact, let $f: \mathbb{R}^{+} \to \mathbb{R}$ satisfy
\begin{align*}
f(s)=\begin{cases}
1, & s \in [0,\frac{1}{2}]\\
\text{non-negative polynomial}, & s \in [\frac{1}{2}, 1]\\
e^{-s}, & s \in [1,\infty)
\end{cases}
\end{align*}
such that $f \in C^2$. Then there exists a generic constant $C_* >0$ such that 
\begin{align*}
|f'(s)| \le C_* f(s),
\end{align*}
Define, for any $R>0$, $
f_R(x)=f(\frac{|x|}{R})$.

  First, according to Lemma \ref{special3}, it is easy to see that for any given $R>0$,
\begin{equation}\label{special6}
\begin{split}
\int\Big(\sum_{l=1}^3 A_l \partial_l\omega+A^*(u)\omega+\frac{\delta-1}{2}\sqrt{a} \phi^{e}\nabla \text{div} u\Big)\cdot \omega |\omega|^2<&\infty,\\
\int\Big(\sum_{l=1}^3 \partial_lA_l +\partial_lA_l  f_R(x)+A_l\partial_lf_R(x)\Big) |\omega|^4<&\infty,\\
\int\Big(|\omega|^4f_R(x)+\omega_t\cdot \omega |\omega|^2+\omega_t \cdot \omega |\omega|^2f_R(x)\Big)<&\infty.
\end{split}
\end{equation}

Second, since the equations \eqref{special2} holds almost everywhere, one can multiply \eqref{special2}  by $4\omega |\omega|^2f_R(x)$ on its both sides and integrate with respect to $x$ over $\mathbb{R}^3$,
then one has
\begin{equation}\label{special7}
\begin{split}
\frac{d}{dt}\int |\omega|^4 f_R(x)=&  \int\Big(\sum_{l=1}^3 \partial_lA_l  f_R(x)+A_l\partial_lf_R(x)\Big) |\omega|^4\\
&-\int 4f_{R}(x)\Big(A^*(u)\omega+\frac{\delta-1}{2}\sqrt{a} \phi^{e}\nabla \text{div} u\Big)\cdot \omega |\omega|^2\\
\leq &C|\nabla u|_\infty \int |\omega|^4 f_R(x)\\
&+C((1+R^{-1})|\omega|_\infty|u|_2+| \phi^{e}\nabla \text{div} u|_2)|\omega|^3_6\\
\leq &C|\nabla u|_\infty \int |\omega|^4 f_R(x)+C(1+R^{-1}),
\end{split}
\end{equation}
for some  constant $C>0$ depending only on $(\rho_0,u_0,\alpha,\beta,A,\gamma,\delta,T)$  but not on $R$,
which, along with  Gronwall’s inequality, implies that
\begin{equation}\label{special8}
\begin{split}
\int |\omega|^4 f_R(x)\leq C(1+R^{-1}) \quad \text{for} \quad t\in [0,T].
\end{split}
\end{equation}

Note that $|\omega|^4 f_R(x)\rightarrow |\omega|^4$  almost everywhere as $R\rightarrow \infty$, then it follows from the Fatou lemma (see Lemma \ref{Fatou}) that
  \begin{equation}\label{special7}
\begin{split}
\int |\omega|^4  \leq \liminf_{R\rightarrow \infty} \int |\omega|^4 f_R(x)\leq C \quad \text{for} \quad t\in [0,T].
\end{split}
\end{equation}
Then \eqref{special5} has been obtained and the desired conclusion follows quickly.

At last,  it follows from the direct calculations that 
 \begin{equation*}
 \begin{split}
 \nabla f=&\frac{a\delta}{\delta-1}\nabla \Big(\rho^{1-\delta} \nabla \rho^{\delta-1} \Big)\\
 =&\frac{a\delta }{\delta-1} \Big(\rho^{1-\delta} \nabla^2 \rho^{\delta-1}-\rho^{2-2\delta} |\nabla \rho^{\delta-1} |^2 \Big) \in L^\infty([0,T];L^2),
 \end{split}
\end{equation*}
where one has used the fact that $\nabla \rho^{\frac{\delta-1}{2}}\in L^\infty([0,T];L^4)$.

\end{proof}

\subsection{Further explanation on the initial assumptions}
Finally, we show some additional information that is implied by  \eqref{th78}-\eqref{th78zx}.
\begin{lemma}\label{further}
$$
\rho^{\delta-1}_0\nabla u_0 \in D^1(\mathbb{R}^3).
$$
\end{lemma}
\begin{proof}
First,  it follows from the   argument used in the proof for Lemma \ref{s4jjk}  that 
  \begin{equation}\label{speciala1}
\begin{split}
\phi^{2e}_0\nabla^2 u_0 \in L^2.
\end{split}
\end{equation}

Second, for any positive constant $\eta>0$, one has 
$$
(\phi_0+\eta)^{2e}\nabla u_0 \in L^6 \quad \text{and} \quad   (\phi_0+\eta)^{2e}\nabla^2 u_0 \in L^2. $$
Then it follows from  the Sobolev imbedding theorem that 
  \begin{equation}\label{speciala2}
\begin{split}
|(\phi_0+\eta)^{2e}\nabla u_0|_6\leq & C(|(\phi_0+\eta)^{2e}\nabla^2 u_0|_2+|\nabla (\phi_0+\eta)^{2e}|_\infty|\nabla u_0|_2)\\
\leq & C(|\phi^{2e}_0\nabla^2 u_0|_2+|\nabla \phi_0^{2e}|_\infty|\nabla u_0|_2)\leq C,
\end{split}
\end{equation}
for some constant $C>0$ that is independent of $\eta>0$.

Moreover, it is very easy to see that, for every $x \in  \mathbb{R}^3$,
$$
(\phi_0+\eta)^{2e}\nabla u_0 \rightarrow \phi^{2e}_0\nabla u_0 \quad \text{a.e.} \quad \text{as} \quad \eta \rightarrow 0,
$$
which, along with Fatou's lemma (see Lemma \ref{Fatou}), implies that 
$$
\int |\phi^{2e}_0\nabla u_0|^6 \leq \liminf_{\eta\rightarrow 0} \int  |(\phi_0+\eta)^{2e}\nabla u_0|^6 <\infty.
$$

Then the proof of this lemma is completed.
\end{proof}

\section{Cauchy problem with  far field vacuum}

This section will be devoted to proving   Theorem \ref{th3s}, and we always assume that \eqref{canshu} holds and $\Omega=\mathbb{R}^3$.   In this section, we only give the proof for \eqref{cri}, and the proof for  \eqref{criinfty} is  similar.     In order to prove (\ref{cri}), we use a contradiction argument. 
Let $( \rho, u)$ be the unique regular solution to  the Cauchy problem  (\ref{eq:1.1})-\eqref{far} with the life span $\overline{T}$ obtained in Theorem 1.1. It is worth pointing out that this solution also  satisfies  the properties stated in Lemmas \ref{5jie}-\ref{further}.
 We assume that
$\overline{T}<+\infty$ and
\begin{equation}\label{we11}
\begin{split}
\displaystyle\lim_{T\mapsto \overline{T}} \left(\sup_{0\le t\le T}\big\|\nabla \rho^{\delta-1}(t,\cdot)\big\|_{L^6(\mathbb{R}^3)}+\int_0^T \|D( u)(t,\cdot)\|_ {L^\infty(\mathbb{R}^3)}\ \text{d}t\right)=C_0<\infty,
\end{split}
\end{equation}
for some constant $C_0>0$.
We will show that under the assumption \eqref{we11}, $\overline{T}$ is actually not the maximal existence time for the regular solution.

By assumptions (\ref{we11}) and the system (\ref{we22}), we first show that  $\rho$ is uniformly bounded.

 \begin{lemma}\label{s1}
\begin{equation*}
\begin{split}
C^{-1}\leq \|\rho\|_{L^\infty([0,T]\times \mathbb{R}^3)}\leq C \quad \text{and} \quad  \|\phi\|_{L^\infty([0,T]; L^q( \mathbb{R}^3))}\leq C\quad \text{for} \quad 0\leq T< \overline{T},
\end{split}
\end{equation*}
where the constant $C>0$ only depends on $(\rho_0,u_0)$,  $C_0$,    $A$, $\gamma$, the constant $q\in [2,+\infty]$ and $\overline{T}$.
 \end{lemma}
 \begin{proof}
First, it is obvious that  $\phi$ can be represented by
\begin{equation}
\label{eq:bb1a}
\phi(t,x)=\phi_0(W(0,t,x))\exp\Big(-(\gamma-1)\int_{0}^{t}\textrm{div}u(s,W(s,t,x))\text{d}s\Big),
\end{equation}
where  $W\in C^1([0,T]\times[0,T]\times \mathbb{R}^3)$ is the solution to the initial value problem
\begin{equation}
\label{eq:bb1z}
\begin{cases}
\displaystyle
\frac{d}{ds}W(s,t,x)=u(s,W(s,t,x)),\quad 0\leq s\leq T,\\[12pt]
W(t,t,x)=x, \quad \ \quad \quad \ \ \ 0\leq s\leq T,\ x\in \mathbb{R}^3.
\end{cases}
\end{equation}
Then it is clear that 
\begin{equation}\label{zqq}|\phi_0|_\infty \exp({-CC_0})\leq \|\phi\|_{L^\infty([0,T]\times \mathbb{R}^3)}\leq |\phi_0|_\infty \exp({CC_0}) \quad \text{for} \quad 0\leq T< \overline{T}.
\end{equation}

Second, it follows   from $(\ref{we22})_1$ that 
\begin{equation}\label{zhumw}
\begin{split}
& \frac{d}{dt}|\phi|_2\leq C |\text{div} u|_\infty|\phi|_2,
\end{split}
\end{equation}
which, along with  Gronwall's inequality and \eqref{zqq},  implies the desired conclusions.

 \end{proof}

 \subsection{Lower order estimates from the strong  singular structure \eqref{we22}}

 Now we give the basic energy estimate.
  \begin{lemma}\label{s2}    
\begin{equation*}
\begin{split}
\sup_{0\leq t\leq T}|u(t)|^2_{ 2}+\int_{0}^{T}|\phi^e\nabla u|^2_{2}\text{\rm d}t\leq C\quad \text{for} \quad 0\leq T<  \overline{T},
\end{split}
\end{equation*}
where the constant $C>0$ only depends on $(\rho_0,u_0)$,  $C_0$,  $\alpha$,  $\beta$, $A$, $\gamma$, $\delta$     and $\overline{T}$.
 \end{lemma}
\begin{proof}  Multiplying $(\ref{we22})_2$ by $2u$ and integrating over $\mathbb{R}^3$, one has
\begin{equation}\label{zhu1}
\begin{split}
& \frac{d}{dt}|u|^2_2+2a\alpha|\phi^e\nabla u|^2_2+2a(\alpha+\beta)|\phi^e\text{div} u|^2_2\\
=&\int  2\Big(-u\cdot \nabla u - \nabla \phi +\delta^{-1}\psi \cdot Q(u)\big)\cdot u \\
 \leq & C\big(|\text{div} u|_\infty|u|^2_2+|\phi|_2|\text{div} u|_2+|\psi|_6|\phi^e\nabla u|_2|u|_3|\phi|^{-e}_\infty\Big)\\
 \leq &     \frac{a\alpha}{2}|\phi^e\nabla  u|^2_2+C\big(|\text{div} u|_\infty|u|^2_2+|u|^2_2+1\big),
\end{split}
\end{equation}
which, along with   Gronwall's inequality, implies that 
\begin{equation}\label{zhu5}
\begin{split}
|u(t)|^2_2+\int_0^t|\phi^e \nabla u(s)|^2_2\text{d}s\leq C\quad \text{for}\quad 0\leq t\leq T.
\end{split}
\end{equation}
\end{proof}

The next lemma provides a  key estimate on $\nabla \phi$ and $\nabla u $.

  \begin{lemma}\label{s4} 
\begin{equation*}
\begin{split}
\sup_{0\leq t\leq T}\Big(|\phi^e \nabla u|^2_{ 2}+|\nabla \phi|^2_{ 2}\Big)(t)+\int_0^T (|\phi^{2e}Lu|^2_2+|\nabla^2 u|^2_2+| u_t|^2_2)\text{\rm d}t\leq C
\end{split}
\end{equation*}
for $ 0\leq T<  \overline{T}$,   where the constant  $C>0$ only depends on $(\rho_0,u_0)$,  $C_0$,  $\alpha$,  $\beta$, $A$, $\gamma$, $\delta$     and $\overline{T}$.
 \end{lemma}
\begin{proof}
Multiplying $(\ref{we22})_2$ by $(-a\phi^{2e}Lu- \nabla \phi )$ and integrating over $\mathbb{R}^3$, one has
\begin{equation}\label{zhu6}
\begin{split}
&\frac{a}{2} \frac{d}{dt}\Big(\alpha|\phi^e\nabla u|^2_2+(\alpha+\beta)|\phi^e\text{div}u|^2_2\Big)+\int (a\phi^{2e}Lu+ \nabla \phi)^2\\
=&-a\alpha\int \phi^{2e} (u\cdot \nabla) u \cdot \nabla \times \text{curl} u+a(2\alpha+\beta)\int \phi^{2e} (u\cdot \nabla) u \cdot  \nabla \text{div}u\\
&-\int  (u\cdot \nabla )u \cdot \nabla \phi -\int u_t \cdot \nabla \phi +\frac{1-\delta}{\delta}\int \psi\cdot Q(u)\cdot u_t\\
&+\frac{a}{2} \int \alpha\big((\phi^{2e})_t|\nabla u|^2+(\alpha+\beta)(\phi^{2e})_t|\text{div} u|^2\big)\\
&+\int \psi\cdot Q(u) \cdot  \nabla \phi+a\int (\psi\cdot Q(u)) \cdot \phi^{2e}Lu
\equiv: \sum_{i=1}^{8} L_i,
\end{split}
\end{equation}
where one has  used the fact that 
$-\triangle u+\nabla\text{div}u=\nabla \times \text{curl} u$.

First, it follows from the standard elliptic estimate shown in Lemma \ref{zhenok} that 
\begin{equation}\label{gaibian}
\begin{split}
&|\nabla^2 u|^2_2-C|\nabla\phi|^2_2
\leq C|-a Lu|^2_2-C|\nabla\phi|^2_2\\
\leq &C|-a \phi^{2e}Lu|^2_2-C|\nabla\phi|^2_2
\leq C|a \phi^{2e}Lu+\nabla\phi|^2_2.
\end{split}
\end{equation}
According to  momentum equations $\eqref{we22}_2$, one can also obtain that 
\begin{equation}\label{ut}
\begin{split}
|u_t|_2\leq & C(|a\phi^{2e}Lu+ \nabla \phi|_2+|u|_3|\nabla u|_6+|\psi|_6|\nabla u|_3)\\
\leq & C\big(|a\phi^{2e}Lu+ \nabla \phi|_2+|\nabla u|^{\frac{1}{2}}_2|\nabla^2 u|_2+|\nabla u|^{\frac{1}{2}}_2|\nabla^2 u|^{\frac{1}{2}}_2\big).
\end{split}
\end{equation}

Second, we start to  estimate the right-hand side of (\ref{zhu6}) term by term.
Due  to
\begin{equation*}
\begin{cases}
\displaystyle
u\times \text{curl}u=\frac{1}{2}\nabla (| u|^2)-u \cdot \nabla u,\\[10pt]
\displaystyle
\nabla \times (a\times b)= (b\cdot \nabla) a-(a\cdot \nabla)b+(\text{div}b)a-  (\text{div}a)b,
\end{cases}
\end{equation*}
 H\"older's inequality, Gagliardo-Nirenberg inequality and Young's inequality,
one gets
\begin{equation}\label{zhu10t}
\begin{split}
|L_1|=&a\alpha\Big|\int  \phi^{2e}(u\cdot \nabla) u \cdot \nabla \times \text{curl} u\Big|\\
=&a \alpha\Big| \int  \text{curl} u\cdot \nabla \times (\phi^{2e} (u\cdot \nabla) u) \Big|\\
=& a \alpha\Big| \int \Big(\phi^{2e} \text{curl} u\cdot \nabla \times ( (u\cdot \nabla) u)+\text{curl} u\cdot \tilde{\textbf{g}}\Big) \Big|\\
=& a \alpha\Big| \int \Big(\phi^{2e} \text{curl} u\cdot \nabla \times ( u\times \text{curl} u)-\text{curl} u\cdot \tilde{\textbf{g}}\Big) \Big|\\
=& a \alpha\Big| \int \phi^{2e}\Big(\frac{1}{2} |\text{curl} u|^2 \text{div}u-\text{curl} u\cdot D(u)\cdot \text{curl} u\Big)\Big| \\
&+ a \alpha\Big| \int\Big(-\frac{1}{2}\nabla \phi^{2e}u |\text{curl} u|^2-\text{curl} u\cdot \tilde{\textbf{g}}\Big) \Big|\\
 \leq&  C|D(u)|_\infty|\phi^{e}\nabla u|^2_2+C|\psi|_6|u|_3|\nabla u |_3|\nabla u|_6\\
 \leq & C|D(u)|_\infty|\phi^{e}\nabla u|^2_2+\epsilon |\nabla^2 u|^2_2+C(\epsilon)|\phi^e \nabla u |^4_2,
 \end{split}
\end{equation} 
where $\tilde{\textbf{g}}=(\tilde{g}_1,\tilde{g}_2,\tilde{g}_3)$ is given by 
\begin{equation*}
\begin{split}
\tilde{g}_1=& \frac{\partial \phi^{2e}}{\partial x_2} (u\cdot \nabla) u^3- \frac{\partial \phi^{2e}}{\partial x_3} (u\cdot \nabla) u^2,  \\
\tilde{g}_2=&\frac{\partial \phi^{2e}}{\partial x_3} (u\cdot \nabla) u^1- \frac{\partial \phi^{2e}}{\partial x_1} (u\cdot \nabla) u^3,  \ \ \ 
\tilde{g}_3=\frac{\partial \phi^{2e}}{\partial x_1} (u\cdot \nabla) u^2- \frac{\partial \phi^{2e}}{\partial x_2} (u\cdot \nabla) u^1,
\end{split}
\end{equation*} 
and similarly, 
\begin{equation}\label{zhu10t3355ss}
\begin{split}
|L_2|=&a(2\alpha+\beta)\Big|\int \phi^{2e} (u\cdot \nabla) u \cdot  \nabla \text{div}u\Big|\\
=& a(2\alpha+\beta)\Big|\int \phi^{2e} \big(-\nabla u: \nabla u^\top  \text{div}u+\frac{1}{2}(\text{div}u)^3\big)\Big|\\
&+ a(2\alpha+\beta)\Big|\int (u\cdot \nabla) u\cdot \nabla \phi^{2e}\text{div}u\Big|\\
\leq & C|\text{div} u|_\infty |\phi^{e}\nabla u|^2_2+C|\psi|_6|u|_3|\nabla u |_3|\nabla  u|_6 \\
 \leq & C|D(u)|_\infty|\phi^{e}\nabla u|^2_2+\epsilon |\nabla^2 u|^2_2+C(\epsilon)|\phi^e \nabla u |^4_2, \\
 \end{split}
\end{equation} 
\begin{equation}\label{zhu10t3355}
\begin{split} 
 |L_3|=&\Big|\int  (u\cdot \nabla) u \cdot \nabla \phi \Big|
=\Big|-\int \nabla u: (\nabla u)^\top \phi -\int  \phi u \cdot \nabla (\text{div}u)  \Big|\\
=&\Big|-\int \nabla u: (\nabla u)^\top\phi +\int  (\text{div}u)^2 \phi +\int  u\cdot\nabla\phi\text{div}u \Big|\\
\leq &C|\phi|^{1-2e}_\infty|\phi^e\nabla u|^2_2+C|\text{div}u|_{\infty}|u|_2|\nabla\phi|_2,\\ 
L_{4}=&-\int u_t \cdot \nabla \phi=\frac{d}{dt} \int \phi \text{div}u -\int  \phi_t \text{div}u \\
=&\frac{d}{dt} \int   \phi \text{div}u +\int u \cdot \nabla \phi \text{div}u +(\gamma-1)\int  \phi (\text{div}u)^2 \\
\leq &\frac{d}{dt} \int  \phi \text{div}u +C|\text{div}u|_{\infty}|u|_2|\nabla\phi|_2+C|\phi|^{1-2e}_\infty|\phi^e\nabla u|^2_2,\\
L_{5}=&\frac{1-\delta}{\delta}\int \psi\cdot Q(u)\cdot u_t
\leq  C|\phi|^{-\frac{e}{2}}_\infty|u_t|_2|\phi^e\nabla u|^{\frac{1}{2}}_2|\nabla^2 u|^{\frac{1}{2}}_2|\psi|_6,\\
L_{6}=&\frac{a}{2} \int \alpha\big((\phi^{2e})_t|\nabla u|^2+(\alpha+\beta)(\phi^{2e})_t|\text{div} u|^2\big)\\
=& \frac{a}{2}\int (-u\cdot \nabla \phi^{2e}-(\delta-1)\phi^{2e} \text{div}u )(\alpha |\nabla u|^2+(\alpha+\beta)|\text{div} u|^2)\\
\leq & C|\psi|_6|u|_3|\nabla u |_3|\nabla u|_6+ C|D(u)|_\infty|\phi^{e}\nabla u|^2_2 \\
 \leq & C|D(u)|_\infty|\phi^{e}\nabla u|^2_2+\epsilon |\nabla^2 u|^2_2+C(\epsilon)|\phi^e \nabla u |^4_2, \\
L_{7}=&\int \psi\cdot Q(u) \cdot  \nabla \phi 
\leq  C|\phi|^{-\frac{e}{2}}_\infty|\nabla \phi |_2|\phi^e\nabla u|^{\frac{1}{2}}_2|\nabla^2 u|^{\frac{1}{2}}_2|\psi|_6\\
\leq & \epsilon |\nabla^2 u|^2_2+C(\epsilon) (|\phi^e \nabla u|^{2}_2+|\nabla \phi|^{2}_2),\\
L_{8}=&\int \psi\cdot Q(u) \cdot \phi^{2e}Lu\\
\leq &C|\phi|^{-\frac{e}{2}}_\infty|\psi|_6|\phi^{2e} Lu |_2|\phi^e \nabla u|^{\frac{1}{2}}_2|\nabla^2 u|^{\frac{1}{2}}_2\\
\leq & C(\epsilon)|\phi^e \nabla u|^2_2+\epsilon (|\nabla^2 u|^{2}_2+|\phi^{2e}Lu|^{2}_2),
\end{split}
\end{equation}
where $\epsilon> 0$ is a sufficiently small constant.

It follows from  (\ref{zhu6})-(\ref{zhu10t3355})  that 
\begin{equation}\label{zhu6qss}
\begin{split}
& \frac{d}{dt}\int  \Big(\frac{a}{2}\alpha|\phi^e\nabla u|^2+\frac{a}{2}(\alpha+\beta)|\phi^e\text{div}u|^2-\phi \text{div}u\Big)+C|\nabla^2 u|^2_2+\frac{a}{2}|\phi^{2e} Lu|^2_2\\
\leq &C(|\phi^e \nabla u|^2_2+|\nabla \phi|^2_2)(1+|D(u)|_\infty+|\phi^e \nabla u|^2_2).
\end{split}
\end{equation}

Second, applying $\nabla$ to  $(\ref{we22})_1$ and multiplying by $(\nabla \phi)^{\top}$,  one has
\begin{equation}\label{zhu20}
\begin{split}
&(|\nabla \phi|^2)_t+\text{div}(|\nabla \phi|^2u)+(\gamma-2)|\nabla \phi|^2\text{div}u\\
=&-2 (\nabla \phi)^\top \nabla u( \nabla \phi)-(\gamma-1) \phi \nabla \phi \cdot \nabla \text{div}u\\
=&-2 (\nabla \phi)^\top D(u) (\nabla \phi)-(\gamma-1) \phi \nabla \phi \cdot \nabla \text{div}u.
\end{split}
\end{equation}
Integrating (\ref{zhu20}) over $\mathbb{R}^3$, one can get 
\begin{equation}\label{zhu21}
\begin{split}
\frac{d}{dt}|\nabla \phi|^2_2
\leq& C(\epsilon)(|D( u)|_\infty+1)|\nabla \phi|^2_2+\epsilon |\nabla^2 u|^2_2.
\end{split}
\end{equation}
Adding (\ref{zhu21}) to (\ref{zhu6qss}), it follows from  H\"older's inequality,  Young's inequality and Gronwall's inequality that 
\begin{equation*}
\begin{split}
|\phi^e \nabla u(t)|^2_{ 2}+|\nabla \phi(t)|^2_{ 2}+\int_0^t (|\nabla^2 u|^2_2+|\phi^{2e} Lu|^2_2)   \text{d}s\leq C\quad \text{for}\quad 0\leq t\leq T,
\end{split}
\end{equation*}
which, together with \eqref{ut}, implies that 
\begin{equation*}
\begin{split}
\int_0^t | u_t|^2_2\text{d}s\leq C\int_0^t (|\phi^{2e} L u|^2_2+|\nabla u|^2_3|u|^2_6+|\nabla \phi|^2_2+|\nabla u|^2_3|\psi|^2_6)\text{d}s\leq C.
\end{split}
\end{equation*}
 \end{proof}

 \begin{lemma}\label{s6h} 
 \begin{equation*}
\begin{split}
\sup_{0\leq t\leq T}\Big(|u_t|^2_2+|u|^2_{D^2}+|\phi^{2e}L u|^2_2\Big)(t)+\int_0^T(|\phi^e \nabla u_t|^2_2+ |u|^2_{D^{2,6}})   \text{\rm d}t \leq C,
\end{split}
\end{equation*}
for $0\leq T < \overline{T}$,  where the constant  $C>0$ only depends on $(\rho_0,u_0)$,  $C_0$,  $\alpha$,  $\beta$, $A$, $\gamma$, $\delta$     and $\overline{T}$.
\end{lemma}
\begin{proof}

It follows from Lemma \ref{zhenok}, equations $\eqref{dege}_3$ and $\eqref{we22}_2$    that 
$$
|u|_{D^2}+|\phi^{2e}L u|_2\leq  C\big(|u_t|_{2}+| u|^2_6|\nabla u|_2+|\nabla \phi|_2+|\psi|^2_{6}|\nabla u|_2\big),
$$
which, along with  Lemmas \ref{s1}-\ref{s4}, implies that 
 \begin{equation}\label{zhu15nnh}
\begin{split}
|u|_{D^2}+|\phi^{2e}L u|_2\leq C(1+|u_t|_2).
\end{split}
\end{equation}

Next, differentiating $(\ref{we22})_3$ with respect to $t$, it reads
\begin{equation}\label{zhu37ssh}
\begin{split}
u_{tt}+a\phi^{2e}Lu_t= -(u\cdot\nabla u)_t -\nabla \phi_t - a\phi^{2e}_tLu+(\psi\cdot Q(u))_t.
\end{split}
\end{equation}
Multiplying (\ref{zhu37ssh}) by $u_t$ and integrating over $\mathbb{R}^3$, one has
\begin{equation}\label{zhu38h}
\begin{split}
&\frac{1}{2} \frac{d}{dt}|u_t|^2_2+a\alpha|\phi^e\nabla u_t|^2_2+a(\alpha+\beta)|\phi^e \text{div} u_t|^2_2\\
=&-\int  \Big((u\cdot\nabla u+\nabla \phi)_t +a\phi^{2e}_tLu-\psi_t\cdot Q(u)-\frac{1}{\delta}\psi\cdot Q(u)_t\Big)\cdot u_t 
\equiv: \sum_{i=9}^{13} L_i.
\end{split}
\end{equation}
We estimate the right-hand side of \eqref{zhu38h} term by term as follows.
\begin{equation}\label{zhu9fgh}
\begin{split}
L_{9}=&-\int  (u\cdot \nabla u)_t \cdot u_t =-\int   \big((u_t \cdot \nabla) u  +(u \cdot \nabla) u_t \Big)\cdot u_t\\
=&-\int  \Big(u_t \cdot D( u) \cdot u_t -\frac{1}{2} ( u_t)^2 \text{div}u \Big)
\leq  C|D(u)|_\infty|u_t|^2_2,\\
\end{split}
\end{equation}
\begin{equation}\label{zhu9fgh-1}
\begin{split}
L_{10}=&-\int    \nabla \phi_t \cdot u_t =\int   \phi_t\text{div}u_t \\
=&-\frac{(\gamma-1)}{2}\frac{d}{dt}\int   \phi (\text{div}u)^2 -\frac{(\gamma-1)}{2}\int   u\cdot \nabla \phi (\text{div}u)^2 \\
&-\frac{(\gamma-1)^2}{2}\int    \phi (\text{div}u)^3 -\int   u\cdot \nabla \phi \text{div}u_t\\
\leq&-\frac{(\gamma-1)}{2}\frac{d}{dt}\int   \phi (\text{div}u)^2 + C|u|_6| \nabla \phi|_6|\nabla u|_2|\nabla u|_6\\
&+C|\phi|_\infty|D( u)|_\infty|\nabla u|^2_2+C|\nabla \phi|_6|u|_3|\nabla u_t|_2 \\
\leq&-\frac{(\gamma-1)}{2}\frac{d}{dt}\int  \phi (\text{div}u)^2 + C(|D( u)|_\infty+|u_t|^2_2+1)+\frac{a\alpha}{4}|\phi^e\nabla u_t|^2_2,\\
L_{11}=&- \int   a\phi^{2e}_tLu \cdot u_t
= \int  a(u\cdot \nabla \phi^{2e}+(\delta-1) \phi^{2e}\text{div}u)Lu \cdot u_t\\
\leq &C|u|_6 |\psi|_6 |Lu|_2|u_t|_6+C|\text{div}u|_\infty|\phi^{2e} Lu|_2|u_t|_2\\
\leq &C(1+|D(u)|_\infty)(|u_t|^2_2+1)+\frac{a\alpha}{4}|\phi^e\nabla u_t|^2_2,\\
L_{12}+& L_{13}=\int  \big(\psi_t\cdot Q(u)+\frac{1}{\delta}\psi\cdot Q(u)_t\big)\cdot u_t\\
=& \int  \frac{1}{\delta}\ \psi \cdot Q(u)_t\cdot u_t-a\delta \int  \phi^{2e}\nabla \text{div}u \cdot Q(u)\cdot u_t\\
&+\int \psi\cdot u \text{div}(Q(u)\cdot u_t)-(\delta-1)\int \text{div}u \psi\cdot Q(u)\cdot u_t\\
\leq& C|\psi|_6(|\nabla u_t|_2|\ u_t|^{\frac{1}{2}}_2|u_t|^{\frac{1}{2}}_6+| u|_6|\nabla^2 u|_2  | u_t|_6+| u|_6|Q(u)|_6  |\nabla u_t|_2)\\
&+C|\psi|_6| u_t|_2|Q(u)|_6  |\nabla u|_6-a\delta \int  \phi^{2e}\nabla \text{div}u \cdot Q(u)\cdot u_t\\
\leq &C(|u_t|^3_2+1)+\frac{a\alpha}{4}|\phi^e\nabla u_t|^2_2+L^*,
\end{split}
\end{equation}
where, via integration by parts, the last term $L^*$ in $L_{12}+L_{13}$ can be estimated as follows:
\begin{equation}\label{zhu9fghgg}
\begin{split}
L^*=&-a\delta \int  \phi^{2e}\nabla \text{div}u \cdot Q(u)\cdot u_t\\
=& a\delta \int ( \phi^{2e} \text{div}u Lu \cdot u_t+ \phi^{2e} \text{div}u Q(u): \nabla u_t+\text{div}u\nabla \phi^{2e}  \cdot Q(u)\cdot u_t ) \\
\leq& C| \phi^{2e}Lu|_2|\text{div}u|_\infty  | u_t|_2+C|\text{div}u|_\infty| \phi^e \nabla u_t|_2  |\phi^e\nabla u|_2\\
&+ C|\psi|_6| u_t|_2|\text{div}u|_6|\nabla u|_6 \\
\leq &C(|u_t|^4_2+1)+\frac{a\alpha}{4}|\phi^e\nabla u_t|^2_2+C|\text{div}u|^2_\infty.
\end{split}
\end{equation}

It follows from  Lemma \ref{zhenok} and equations  $\eqref{dege}_3$ that
 \begin{equation}\label{zhu55}
\begin{split}
|\nabla^2 u|_6 \leq& C|\varphi|_\infty(|u_t|_6+|u\cdot \nabla u|_6+|\nabla \phi|_6+|\psi \cdot Q( u)|_6)\\
\leq& C(1+|\nabla u_t|_2+|\nabla \phi|_6+ |\nabla u|^{\frac{1}{2}}_6 |\nabla^2 u|^{\frac{1}{2}}_6),
\end{split}
\end{equation}
where one has  used the fact that $|\nabla u|_\infty\leq C|\nabla u|^{\frac{1}{2}}_6 |\nabla^2 u|^{\frac{1}{2}}_6$. Then, according to the  Young's inequality, one has 
 \begin{equation}\label{zhu65sss}
\begin{split}
|\nabla^2 u|_6
\leq C(1+|\nabla u_t|_2+|u_t|_2),
\end{split}
\end{equation}
which,  along with \eqref{zhu9fghgg},   implies that 
\begin{equation}\label{lstar}
\begin{split}
|\text{div}u|^2_\infty\leq & C(1+|u_t|^2_2) +\frac{a\alpha}{4}|\phi^e\nabla u_t|^2_2,\\
L^*
\leq &C(|u_t|^4_2+1)+\frac{1}{4}a\alpha|\phi^e\nabla u_t|^2_2.
\end{split}
\end{equation}

It is clear from (\ref{zhu38h})-\eqref{lstar} and  \eqref{we11} that
\begin{equation}\label{zhu8qqssh}
\begin{split}
& \frac{d}{dt}(|u_t|^2_2+|\sqrt{\phi}\text{div}u|^2_2)+| \phi^{e}\nabla u_t|^2_2
\leq  C(1+|D(u)|_\infty+|u_t|^2_2)(|u_t|^2_2+1).
\end{split}
\end{equation}
Integrating (\ref{zhu8qqssh}) over $(\tau,t)$ $(\tau \in( 0,t))$, one has
\begin{equation}\label{zhu13vbnh}
\begin{split}
&|u_t|^2_2+|\sqrt{\phi} \text{div}u(t)|^2_2+\int_\tau^t| \phi^{e}\nabla u_t(s)|^2_2\text{d}s\\
\leq&  |u_t(\tau)|^2_2+C\int_\tau^t (1+|D(u)|_\infty+|u_t|^2_2)(|u_t|^2_2+1)\text{d}s.
\end{split}
\end{equation}

It follows from   momentum equations $ (\ref{we22})_2$ that
\begin{equation}\label{zhu15h}
\begin{split}
|u_t(\tau)|_2\leq C\big( |u|_\infty |\nabla u|_2+|\nabla \phi|_2+|\phi^{2e}L u|_{2}+|\psi|_\infty|\nabla u|_2\big)(\tau),
\end{split}
\end{equation}
which, along with the definition of the regular solution and the  assumption (\ref{th78zx}), implies that
\begin{equation*}
\begin{split}
\lim \sup_{\tau\rightarrow 0}|u_t(\tau)|_2
\leq &C\big( |u_0|_\infty |\nabla u_0|_2+|\nabla \phi_0|_2+|g_2|_{2}+|\psi_0|_\infty|\nabla u_0|_2\big)
\leq  C.
\end{split}
\end{equation*}
Letting $\tau \rightarrow 0$ in (\ref{zhu13vbnh}),  applying Gronwall's inequality, we arrive at
$$
|u_t(t)|^2_2+|\phi^{2e}Lu(t)|_{2}+|u(t)|^2_{D^2}+\int_0^t (|\phi^e \nabla u_t|^2_2+ |u|^2_{D^{2,6}}) \text{d}s\leq C\quad \text{for} \quad 0\leq t\leq T.
$$
This completes the proof of this lemma.
\end{proof}

 \begin{lemma}\label{phit} 
 \begin{equation}\label{zhu14ssh}
\begin{split}
&\sup_{0\leq t\leq T}\big(\|g\|_{H^1\cap D^{1,6}}+\|(\phi_t,g_t,\varphi_t)\|^2_{L^2(\mathbb{R}^3)\cap L^6(\mathbb{R}^3)}\big)(t)\leq C\quad \text{for} \quad 0\leq T<  \overline{T},
\end{split}
\end{equation}
where the constant  $C>0$ only depends on $(\rho_0,u_0)$,  $C_0$,  $\alpha$,  $\beta$, $A$, $\gamma$, $\delta$     and $\overline{T}$.
\end{lemma}
The proof of this lemma can be directly obtained from the equations $\eqref{dege}_1$, $\eqref{dege}_2$ and $\eqref{we22}_1$,  and the conclusions of Lemmas \ref{s1}-\ref{s6h}. Here we omit its details.
 \subsection{Higher order estimates from  the ``Degenerate"-``Weak-Singular"   structure \eqref{dege}}

Lemma \ref{s6h}   implies that
 \begin{equation}\label{keygradient}
\int_0^t|\nabla u(\cdot, s)|^2_\infty \text{d}s\leq C\quad \text{for} \quad 0\leq t < \overline{T}, \end{equation}
where  $C>0$ is some  finite constant.  Noting that \eqref{we22} is essentially a hyperbolic-singular parabolic  system, it is very hard to   derive other higher order estimates for the regularity of
the regular solutions via using this structure directly. Indeed, we need to ask for help from the so-called  "Degenerate"--"Weak-Singular"   structure \eqref{dege}, which will be  shown  in the following 3 lemmas.

\begin{lemma}\label{s8}  
\begin{equation*}
\begin{split}
&\sup_{0\leq t\leq T}\big(\|(g,\phi)\|^2_{D^2}+|f|^2_{D^1}+\|(g_t,\phi_t)\|^2_1+|f_t|^2_{2}\big)(t)\\
&\quad+\int_{0}^{T}\Big(|u|^2_{D^{3}}+|\phi_{tt}|^2_{2}+|g_{tt}|^2_{2}\Big)\text{\rm d}t\leq C\quad \text{for} \quad 0\leq T<  \overline{T},
\end{split}
\end{equation*}
where the constant  $C>0$ only depends on $(\rho_0,u_0)$,  $C_0$,  $\alpha$,  $\beta$, $A$, $\gamma$, $\delta$     and $\overline{T}$.
 \end{lemma}
 \begin{proof}

First, it follows from equations $\eqref{dege}_3$ and Lemma \ref{zhenok} that 
 \begin{equation}\label{zhu150}
\begin{split}
|u|_{D^3}\leq& C(|\varphi u_t|_{D^1}+|\varphi u\cdot \nabla u|_{D^1}+| \nabla g|_{D^1}+|f \cdot Q( u)|_{D^1})\\
\leq& C(1+|u_t|_{D^1}+| \phi|_{D^2}+ |\nabla^2 u|_3+|\nabla f|_2 |\nabla u|_\infty).
\end{split}
\end{equation}
where one has used the following relation 
\begin{equation}\label{gphi}
 \nabla^2 g=C_1\rho^{1-\delta} \nabla^2 \phi+C_2 \nabla \varphi\otimes \nabla \phi,
\end{equation}
for two constants $C_1>0$ and $C_2>0$.
Then it follows from  Young's inequality that 
 \begin{equation}\label{zhu1500}
|u|_{D^3}
\leq C(1+|u_t|_{D^1}+| \phi|_{D^2}+|\nabla f|_2).
\end{equation}

Next, applying $\nabla$ to $(\ref{dege})_4$, multiplying the resulting equations  by $2\nabla f$ and integrating over $\mathbb{R}^3$,  then according to \eqref{zhu1500}, one has
\begin{equation}\label{zhenzhen11}\begin{split}
\frac{d}{dt}|\nabla f|^2_2\leq& C\big|\nabla u\big|_\infty|\nabla f|^2_2+C|\nabla^3 u|_2|\nabla f|_2+|\nabla^2 u|_3 |f|_6|\nabla f|_2\\
\leq& C(1+|\nabla u|_\infty)|f|^2_{D^1}+C(1+|\phi|^2_{D^2}+|u_t|^2_{D^1}).
\end{split}
\end{equation}

On the other hand, applying $\nabla^2$ to $(\ref{we22})_1$,  multiplying the resulting equations by $ 2\nabla^2 \phi$ and integrating  over ${\mathbb {R}}^3$, one has 
\begin{equation}\label{zhuu12acc}
\begin{split}
\frac{d}{dt}|\phi|^2_{D^2}
\leq& C|\nabla u|_\infty|\phi|^2_{D^2}+C|\nabla \phi|_6|\phi|_{D^2}|\nabla^2 u|_3+C|\phi|_\infty|\phi|_{D^2}|\nabla^3 u|_2\\
\leq& C(1+|\nabla u|_\infty)(|\phi|^2_{D^2}+|f|^2_{D^1})+C(1+|\nabla u_t|^2_2),
\end{split}
\end{equation}
which, together with \eqref{zhenzhen11}, gives that
\begin{equation}\label{zhmm12acc}
\begin{split}
\frac{d}{dt}(|\phi|^2_{D^2}+|f|^2_{D^1})
\leq& C(1+|\nabla u|_\infty)(|\phi|^2_{D^2}+|f|^2_{D^1})+C(1+|\nabla u_t|^2_2).\end{split}
\end{equation}
Then  it follows from Gronwall's inequality, \eqref{zhmm12acc} and \eqref{keygradient}  that 
\begin{equation}\label{yijieguji}
\begin{split}
|\phi(t)|^2_{D^2}+|f(t)|^2_{D^1}+\int_{0}^{t}|u(s)|^2_{D^{3}}\text{d}s\leq C\quad \text{for} \quad 0\leq t\leq T.
\end{split}
\end{equation}

Finally, according to \eqref{gphi} and   the following relations
\begin{equation}\label{ghtuss}
\begin{split}
f_t=&-\nabla (u \cdot f)-\nabla \text{div} u,\quad\phi_t=-u\cdot \nabla \phi-(\gamma-1)\phi\text{div} u,\\
g_t=&-u\cdot \nabla g-(\gamma-\delta) g \text{div} u,\quad \phi_{tt}=-(u\cdot \nabla \phi)_t-(\gamma-1)(\phi \text{div} u)_t,\\
g_{tt}=&-(u\cdot \nabla g)_t-(\gamma-\delta) (g \text{div} u)_t,
\end{split}
\end{equation}
we conclude the proof of this lemma.
\end{proof}

In order to obtain higher order regularity, we need the following improved estimate.
\begin{lemma}\label{s9} 
\begin{equation*}
\begin{split}
\sup_{0\leq t\leq T}\big(|u_t|^2_{D^1}+|u|^2_{D^3}\big)(t)+\int_{0}^{T}(|\sqrt{\varphi}u_{tt}|^2_2+|u_t|^2_{D^2})\text{\rm d}t\leq C\quad \text{for} \quad 0\leq T<  \overline{T},
\end{split}
\end{equation*}
where the constant  $C>0$ only depends on $(\rho_0,u_0)$,  $C_0$,  $\alpha$,  $\beta$, $A$, $\gamma$, $\delta$     and $\overline{T}$.
 \end{lemma}
\begin{proof} First, 
 \begin{equation}\label{elliptictt}
aLu_t=-\varphi u_{tt} -\varphi (u\cdot\nabla u)_t -\varphi_t (u_t+u\cdot\nabla u)-\nabla g_t+(f\cdot Q(u))_t,
\end{equation}
  Lemmas \ref{zhenok}  and \ref{phit} yield
 \begin{equation}\label{zhu150z}
\begin{split}
|u_t|_{D^2}\leq& C(|\varphi u_{tt}|_2+ |\varphi (u\cdot\nabla u)_t|_2+|\varphi_t u_t|_2 +|\varphi_t u\cdot\nabla u|_2\\[6pt]
&+|\nabla g_t|_2+|(f\cdot Q(u))_t)|_2)\\
\leq& C(1+|\sqrt{\varphi}u_{tt}|_{2}+|\nabla u_t|_2+|\nabla u_t|_3+|\nabla u|_\infty),
\end{split}
\end{equation}
which implies, with the help of Young's inequality, that
 \begin{equation}\label{zhu1500x}
|u_t|_{D^2}
\leq C(1+|\sqrt{\varphi}u_{tt}|_{2}+|\nabla u_t|_2+|u|_{D^{2,6}}).
\end{equation}

Now, multiplying  $(\ref{zhu37ssh})$ by $u_{tt}$ and integrating over $\mathbb{R}^3$,  one has
\begin{equation}\label{zhu19ss}
\begin{split}
& \frac{a}{2}\frac{d}{dt}\Big(\alpha|\nabla u_t|^2_2+(\alpha+\beta)|\text{div}u_t|^2_2\Big)+| \sqrt{\varphi}u_{tt}|^2_2\\
=&\int  \Big(-\varphi (u\cdot \nabla u)_t-\varphi_t(u_t+u\cdot \nabla u)-\nabla g_t+(f \cdot Q(u))_t\Big)\cdot u_{tt}  \equiv: \sum_{i=14}^{18} L_i.
\end{split}
\end{equation}
For the terms $L_{14}$--$L_{18}$, we perform the following estimates:
\begin{equation}\label{zhu201xc}
\begin{split}
L_{14}=&-\int  \varphi (u\cdot \nabla u)_t \cdot u_{tt}\\
 \leq  & C|\varphi|^{\frac{1}{2}}_\infty(|u_t|_{6}|\nabla u|_{3}+ |u|_\infty |\nabla u_t|_2)|\sqrt{\varphi}u_{tt}|_2\\
 \leq & C|\nabla u_t|^2_2+\frac{1}{10}|\sqrt{\varphi}u_{tt}|^2_2,\\ 
L_{15}=&-\int  \varphi_t u_t \cdot u_{tt} 
= \int (u\cdot \nabla \varphi+(1-\delta)\varphi \text{div}u) u_t \cdot u_{tt} \\
 \leq & C\big(|\varphi|^{\frac{3}{2}}_\infty|u|_{6}|\psi|_{6}+|\varphi|^{\frac{1}{2}}_\infty  |\nabla u|_3\big)|u_t|_{6} |\sqrt{\varphi}u_{tt}|_2\\
  \leq  & C|\nabla u_t|^2_2+\frac{1}{10}|\sqrt{\varphi}u_{tt}|^2_2,\\  
  L_{16}=&-\int  \varphi_t (u\cdot \nabla) u \cdot u_{tt} \\
= & \int (u\cdot \nabla \varphi+(1-\delta)\varphi \text{div}u) (u\cdot \nabla) u \cdot u_{tt} \\
\leq & C\big(|\varphi|^{\frac{3}{2}}_\infty|u|_\infty|\psi|_{6} +|\varphi|^{\frac{1}{2}}_\infty|\nabla u|_6\big) |u|_6 |\nabla u|_6|\sqrt{\varphi}u_{tt}|_2\leq C|\sqrt{\varphi}u_{tt}|_2,\\
L_{17}=&-\int  \nabla g_t \cdot u_{tt} 
=\frac{d}{dt}\int   g_t  \text{div}u_{t} -\int  g_{tt}  \text{div}u_{t} \\
\leq&\frac{d}{dt}\int  g_t \text{div}u_{t} + C|\nabla u_t|_2|g_{tt}|_2
\leq  \frac{d}{dt}\int  g_t \text{div}u_{t} + C(|\nabla u_t|^2_2+|g_{tt}|^2_2),\\  
 \end{split}
\end{equation}
\begin{equation}\label{zhu201}
\begin{split} 
L_{18}=&\int  (f \cdot Q(u))_t\cdot u_{tt}
=\int  f \cdot Q(u)_t\cdot u_{tt}+\int  f_t \cdot Q(u)\cdot u_{tt}\\
\leq& C|\varphi|^{\frac{1}{2}}_\infty|\psi|_6 |\nabla u_t|_3 |\sqrt{\varphi} u_{tt}|_2+\frac{d}{dt}\int  f_t \cdot Q(u)\cdot u_{t}  \\
&-\int  f_{tt} \cdot Q(u)\cdot u_{t} -\int  f_t \cdot Q(u_t)\cdot u_{t} \\
\leq & C|\varphi|^{\frac{1}{2}}_\infty|\psi|_6 |\nabla u_t|_3 |\sqrt{\varphi} u_{tt}|_2+C|f_t|_2|\nabla u_t|_6|u_t|_3+\frac{d}{dt}\int  f_t \cdot Q(u)\cdot u_{t}\\
&-\int  (u\cdot f)_t \text{div}(Q(u)\cdot u_{t})+\int  \nabla \text{div}u_t \cdot Q(u)\cdot u_{t}\\
\leq & \frac{d}{dt}\int  f_t \cdot Q(u)\cdot u_{t}+ C\big(|\varphi|^{\frac{1}{2}}_\infty|\psi|_6 |\nabla u_t|_3 |\sqrt{\varphi} u_{tt}|_2+|f_t|_2|\nabla u_t|_6|u_t|_3\big)\\
&+C|u_t|_6|f|_6(|\nabla^2 u|_2|u_t|_6+|\nabla u|_6|\nabla u_t|_2)\\
&+C|u|_\infty|f_t|_2(|\nabla^2 u|_3|u_t|_6+|\nabla u|_\infty|\nabla u_t|_2)+C|\nabla^2 u_t|_2|\nabla u|_3|u_t|_6\\
\leq &\frac{d}{dt}\int  f_t \cdot Q(u)\cdot u_{t}+\frac{1}{10}|\sqrt{\varphi}u_{tt}|^2_2+C(1+|\nabla u_t|^2_2+|\nabla^2 u|^2_6).
\end{split}
\end{equation}
Therefore, \eqref{zhu19ss}-\eqref{zhu201} imply that
\begin{equation}\label{zhu19qqss}
\begin{split}
& \frac{a}{2}\frac{d}{dt}\Big(|\nabla u_t|^2_2+(\alpha+\beta)|\text{div}u_t|^2_2-\int  (g_t \cdot \text{div}u_{t}+ f_t \cdot Q(u)\cdot u_{t})\Big)+|\sqrt{\varphi} u_{tt}|^2_2\\
\leq& C(1+|\nabla u_t|^2_2+|\nabla^2 u|^{2}_6+|g_{tt}|^2_2),
\end{split}
\end{equation}
which, upon integrating over $(\tau,t)$, along with Lemma \ref{s8},  yields
\begin{equation}\label{liu03}
\begin{split}
|\nabla u_t(t)|^2_2+\int_{\tau}^t| \sqrt{\varphi}u_{tt}(s)|^2_2\text{d}s\leq C +|\nabla u_t(\tau)|^2_2, \quad  0\leq t \leq T,
\end{split}
\end{equation}
where one has  used the fact that for any $\epsilon>0$,
\begin{equation}\label{liu01}
\begin{split}
\int   \big(g_t \cdot \text{div}u_{t}+ f_t \cdot Q(u)\cdot u_{t}\big) 
\leq& \epsilon|\nabla u_t|^2_2+C.
\end{split}
\end{equation}
On the other hand,  it follows from the momentum equations $ (\ref{we22})_2$ that
\begin{equation}\label{zhu15wsx}
\begin{split}
|\nabla u_t(\tau)|_2\leq& \big( |\nabla \big(u\cdot \nabla u+ \nabla \phi+ a \phi^{2e} Lu-\psi\cdot Q(u)\big)|_{2}\big)(\tau).
\end{split}
\end{equation}
Then by the assumption (\ref{th78zx}), one has
\begin{equation}\label{zhu15vbwsx}
\begin{split}
\lim \sup_{\tau\rightarrow 0}|\nabla u_t(\tau)|_2\leq &C( |\nabla (u_0\cdot \nabla u_0)|_2+|\nabla^2\phi_0|_2)\\
&+C(|\nabla (\psi_0\cdot Q(u_0))|_2+| \rho^{\frac{1-\delta}{2}}_0g_3|_{2})\leq C.
\end{split}
\end{equation}

Letting $\tau \rightarrow 0$ in (\ref{liu03}), we finally prove that
\begin{equation}\label{zhu22wsxuy}
\begin{split}
|\nabla u_t(t)|^2_2+\int_{0}^t| \sqrt{\varphi}u_{tt}(s)|^2_2\text{d}s\leq C, \quad  0\leq t \leq T.
\end{split}
\end{equation}
The rest of desired estimates follows quickly  from  (\ref{zhu1500}),  (\ref{zhu1500x}) and \eqref{zhu22wsxuy}.

\end{proof}

It remains to prove the following lemma for the required regularity estimate.

\begin{lemma}\label{s10} 
\begin{equation*}
\begin{split}
\sup_{0\leq t\leq T}(|\phi|^2_{D^3}+|f|^2_{D^2}+| \varphi|^2_{D^{2,3}}+\|\phi_t\|^2_2+|\phi_{tt}|_2)(t)&\\
+\sup_{0\leq t\leq T}|f_t(t)|^2_{D^1}+\int_{0}^{T}(|u|^2_{D^{4}}+|\varphi_{tt}|^2_2)\text{\rm d}t\leq &C \quad \text{for} \quad 0\leq T<  \overline{T},
\end{split}
\end{equation*}
where the constant  $C>0$ only depends on $(\rho_0,u_0)$,  $C_0$,  $\alpha$,  $\beta$, $A$, $\gamma$, $\delta$     and $\overline{T}$.
 \end{lemma}
 \begin{proof} 
First, for the $D^2$ estimate of $\varphi$, it follows from direct calculations that 
\begin{equation*}
\begin{split}
\nabla^2 \varphi=&-2e\nabla (\phi^{-2e}\nabla \log \phi)
=4e^2 \phi^{-2e} \nabla \log \phi \otimes  \nabla \log \phi-2e \phi^{-2e} \nabla^2 \log \phi,
\end{split}
\end{equation*}
which, along with Lemma \ref{s1} and \eqref{yijieguji}, implies that 
$$
|\nabla^2 \varphi|_3\leq C(|f|^2_6|\phi^{-2e}|_\infty+|\phi^{-2e}|_\infty |\nabla^2 \log \phi|_3)\leq C(1+|f|^{\frac{1}{2}}_{D^2}) \quad \text{for} \quad 0\leq t\leq T.
$$
 
It follows from equations $\eqref{dege}_3$, Lemma \ref{zhenok} and the relation \eqref{gphi} that 
 \begin{equation}\label{zhu150jk}
\begin{split}
|u|_{D^4}\leq& C(|\varphi u_t|_{D^2}+|\varphi u\cdot \nabla u|_{D^2}+| \nabla g|_{D^2}+|f \cdot Q( u)|_{D^2})\\
\leq& C(1+|u_t|_{D^2}+| \phi|_{D^3}+|\nabla^2 f|_2).
\end{split}
\end{equation}

Next, applying $\nabla^2$ to $(\ref{dege})_4$, multiplying the resulting equations  by $2\nabla^2 f$ and then  integrating over $\mathbb{R}^3$,  one has
\begin{equation}\label{zhenzhen11jk}\begin{split}
\frac{d}{dt}|\nabla^2 f|^2_2\leq& C\big(| f|^2_{D^2}+|\nabla^4 u|_2|f|_{D^2}+1\big)\\
\leq & C\big(| f|^2_{D^2}+|\phi|_{D^3}+|u_t|^2_{D^2}+1\big).
\end{split}
\end{equation}

On the other hand,applying $\nabla^3$ to $(\ref{we22})_1$,  multiplying the resulting equations by $ 2\nabla^3\phi$ and then  integrating over ${\mathbb {R}}^3$, one has 
\begin{equation}\label{zhuu12accjk}
\begin{split}
\frac{d}{dt}|\phi|^2_{D^3}
\leq& C\big(|\phi|^2_{D^3}+|\nabla^4 u|_2|\phi|_{D^3}+1\big)\\
\leq & C\big(|\phi|^2_{D^3}+|u_t|^2_{D^2}+| f|^2_{D^2}+1\big),
\end{split}
\end{equation}
which, together with \eqref{zhu150jk}-\eqref{zhenzhen11jk}, gives that
\begin{equation}\label{yijiegujijk}
\begin{split}
|\phi(t)|^2_{D^3}+|f(t)|^2_{D^2}+\int_{0}^{t}|u|^2_{D^{4}}\text{d}s\leq C\quad \text{for} \quad 0\leq t\leq T.
\end{split}
\end{equation}

Finally, the rest of the estimates follows from the relation \eqref{ghtuss} and 
$$
\varphi_{tt}=-u_t\cdot \nabla \varphi-u\cdot \nabla \varphi_t-(1-\delta)\varphi_t\text{div} u-(1-\delta)\varphi\text{div} u_t.
$$
\end{proof}

 \subsection{Higher  order estimates from  the strong  singular   structure \eqref{we22}}
\begin{lemma}\label{3} 
\begin{equation*}
\begin{split}
\text{ess}\sup_{0\leq t \leq T}\big(\|\psi\|^2_{D^1\cap D^2}+|\psi_t|^2_{D^1}+ |\phi^{2e}\nabla^2u|^2_{D^1} \big)(t)+\int_{0}^{T}|\phi^{2e}\nabla^2u|^2_{D^2}\text{\rm d}t\leq C
\end{split}
\end{equation*}
for $ 0\leq T <\overline{T}$,  where the constant  $C>0$ only depends on $(\rho_0,u_0)$,  $C_0$,  $\alpha$,  $\beta$, $A$, $\gamma$, $\delta$     and $\overline{T}$.

\end{lemma}
\begin{proof}

First,  set $\varsigma=(\varsigma_1,\varsigma_2,\varsigma_3)^\top$ ($1\leq |\varsigma|\leq 2$ and $\varsigma_i=0,1,2$). Applying  $\partial_{x}^{\varsigma} $ to $(\ref{we22})_3$,
multiplying by $2\partial_{x}^{\varsigma} \psi$ and then integrating over $\mathbb{R}^3$, one can get
\begin{equation}\label{zhenzhen}\begin{split}
\frac{d}{dt}|\partial_{x}^{\varsigma}  \psi|^2_2
\leq & \Big(\sum_{l=1}^{3}|\partial_{l}A_l|_\infty+|B|_\infty\Big)|\partial_{x}^{\varsigma}  \psi|^2_2+|\Theta_\varsigma |_2|\partial_{x}^{\varsigma}  \psi|_2,
\end{split}
\end{equation}
where
\begin{equation*}
\begin{split}
\Theta_\varsigma=\partial_{x}^{\varsigma} (B \psi)-B\partial_{x}^{\varsigma}  \psi+\sum_{l=1}^{3}\big(\partial_{x}^{\varsigma} (A_l \partial_l \psi)-A_l \partial_l\partial_{x}^{\varsigma}  \psi\big)+ a\delta \partial_{x}^{\varsigma}\big(\phi^{2e}\nabla \text{div} u\big).
\end{split}
\end{equation*}

For $|\varsigma|=1$, it is easy  to obtain
\begin{equation}\label{zhen2}
\begin{split}
|\Theta_\varsigma |_2\leq &C\big(|\nabla^2 u|_2|\psi|_\infty+|\nabla u|_\infty|\nabla \psi|_2+|\phi^{2e}\nabla^2 u|_{D^1}\big).
\end{split}
\end{equation}
Similarly, for $|\varsigma|=2$, one has
\begin{equation}\label{zhen2}
\begin{split}
|\Theta_\varsigma |_2\leq& C\big(|\nabla u|_\infty|\nabla^2\psi|_2+|\nabla^2 u|_3|\nabla\psi|_6+|\nabla^3 u|_2|\psi|_\infty+|\phi^{2e}\nabla^2 u|_{D^2}\big).
\end{split}
\end{equation}
It follows from
(\ref{zhenzhen})-(\ref{zhen2}) and the Gagliardo-Nirenberg inequality that
\begin{equation}\label{zuihou}
\frac{d}{dt}\|\psi(t)\|_{D^1\cap D^2}\leq C\|\psi(t)\|_{D^1\cap D^2}+C\|\phi^{2e}\nabla^2 u\|_{D^1\cap D^2}.
\end{equation}

Second, for the estimates of $\|\phi^{2e}\nabla^2 u\|_{D^1\cap D^2}$, it follows from direct calculations and Lemmas \ref{3jie}-\ref{4jie} that 
\begin{equation}\label{singular1}
\begin{split}
|\phi^{2e}\nabla^2u|_{D^1}\leq & C(|\phi^{2e}\nabla u|_{D^2}+|\psi|_6|\nabla^2u|_3+|\nabla u|_\infty|\nabla \psi|_2)\\
\leq & C(1+\|\psi(t)\|_{D^1\cap D^2}+|H|_{D^1}+|fH|_2+|G(\psi,\partial_k u)|_2)\\
\leq &  C(1+\|\psi(t)\|_{D^1\cap D^2}),\\
|\phi^{2e} \nabla^2 u|_{D^2}
\leq& C\big(|\nabla^2 H|_2+|f \cdot  H|_{D^1}+||f|^2H|_2+|G(\nabla \phi^{2e},\nabla^2 u )|_2 \big)\\
\leq & C\big(1+|u_t|_{D^2}+\|\psi(t)\|_{D^1\cap D^2}\big).
\end{split}
\end{equation}

According to \eqref{zhu1500x}, one has 
 \begin{equation}\label{zhu150011}
|u_t|_{D^2}
\leq C(1+|\sqrt{\varphi}u_{tt}|_{2}+|\nabla u_t|_2+|u|_{D^{2,6}})\leq C(1+|\sqrt{\varphi}u_{tt}|_{2}),
\end{equation}
which, together with \eqref{zuihou}-\eqref{singular1}, implies that 
\begin{equation*}
\frac{d}{dt}\|\psi(t)\|_{D^1\cap D^2}\leq C\|\psi(t)\|_{D^1\cap D^2}+C(1+|\sqrt{\varphi}u_{tt}|_{2}).
\end{equation*}
Then according to  Gronwall's inequality, one has 
 \begin{equation}\label{zhu150022}
\|\psi(t)\|_{D^1\cap D^2}\leq C\Big(1+\int_0^t |\sqrt{\varphi}u_{tt}|_{2} \text{d}s\Big)\exp (Ct)\leq C(T).
\end{equation}

Second, according to  equations $(\ref{we22})_3$,
for  $0\leq t \leq T$, it holds that
\begin{equation}\label{uu3}
\begin{split}
|\nabla \psi_t(t)|_{2}\leq & C \big(|u|_{\infty}| \psi|_{D^2}+|\nabla u|_\infty|\nabla \psi|_2+|\nabla^2 u|_3|\psi|_{6}+|\phi^{2e}\nabla^2 u|_{D^1}\big)\leq C.
\end{split}
\end{equation}

\end{proof}

\begin{lemma}\label{laosi}
$$
\sup_{0\leq t\leq T}\big(|\phi^e\nabla^2 u|_2+  |\omega|_4)(t)\leq C \quad \text{for} \quad 0\leq  t \leq T.
$$
where the constant  $C>0$ only depends on $(\rho_0,u_0)$,  $C_0$,  $\alpha$,  $\beta$, $A$, $\gamma$, $\delta$     and $\overline{T}$.
\end{lemma}
\begin{proof}
First,  according to Lemmas \ref{5jie} and  \ref{s9}-\ref{3}, one has 
\begin{equation*}|\phi^{e}\partial_k u|_{D^2}\leq C(|\phi|^{-e}_\infty|\partial_kH|_2+|\partial_k \phi^{-e}H|_2+|G( \phi^{-e}\psi,\partial_ku)|_2)\leq C,
\end{equation*}
Then it follow from  Lemmas \ref{s4} and \ref{gag111} that
$$
\|\phi^{e}\partial_k u\|_{1} \leq C \|\phi^{e}\partial_k u\|^{\frac{1}{2}}_0 \|\phi^{e}\partial_k u\|^{\frac{1}{2}}_2,
$$
which,  implies that 
\begin{equation}\label{decom4}
|\phi^{e}\nabla u|_{D^1}+|\phi^{e}\nabla^2 u|_{2}\leq C \quad \text{for} \quad 0\leq  t \leq T.
\end{equation}

Second, since the equations \eqref{special2} holds almost everywhere, one can multiply \eqref{special2}  by $4\omega |\omega|^2$ on its both sides and integrate with respect to $x$ over $\mathbb{R}^3$,
then one has
\begin{equation}\label{special7uyy}
\begin{split}
\frac{d}{dt}\int |\omega|^4 =&  \int \sum_{l=1}^3 \partial_lA_l   |\omega|^4-\int 4\Big(A^*(u)\omega+\frac{\delta-1}{2}\sqrt{a} \phi^{e}\nabla \text{div} u\Big)\cdot \omega |\omega|^2\\
\leq &C|\nabla u|_\infty  |\omega|^4_4 +C|\omega|^3_6|\phi^{e}\nabla \text{div} u|_2\leq C(|\omega|^4_4+1),
\end{split}
\end{equation}
which, along with  Gronwall’s inequality, implies the desired estimate.

\end{proof}

\begin{lemma}\label{5zx} 
\begin{equation*}
\begin{split}
\text{ess}\sup_{0\leq t\leq T}\big(|\phi^e\nabla u_t|^2_2+  |\phi^e \nabla \big(\phi^{2e}Lu\big)|^2_2)(t)+\int_{0}^{T}\Big(|u_{tt}|^2_{2}+|\phi^{2e}Lu_t|^2_{2}\Big)\text{\rm d}t\leq& C,
\end{split}
\end{equation*}
for $0\leq T < \overline{T}$,  where the constant  $C>0$ only depends on $(\rho_0,u_0)$,  $C_0$,  $\alpha$,  $\beta$, $A$, $\gamma$, $\delta$     and $\overline{T}$.
 \end{lemma}
\begin{proof}
Multiplying  (\ref{zhu37ssh}) by $u_{tt}$ and integrating over $\mathbb{R}^3$ give
\begin{equation}\label{zhu19}
\begin{split}
& \frac{1}{2}\frac{d}{dt}\Big(a\alpha|\phi^e\nabla u_t|^2_2+a(\alpha+\beta)|\phi^e\text{div} u_t|^2_2\Big)+| u_{tt}|^2_2\\
=&\int \Big(-(u\cdot \nabla u)_t-\nabla \phi_t-a \phi^{2e}_t Lu -a \nabla \phi^{2e} \cdot Q(u_t) \Big)\cdot u_{tt} \\
&+\int \Big(\frac{a}{2} \phi^{2e}_t\big(\alpha|\nabla u_t|^2+(\alpha+\beta)|\text{div}u_t|^2\big)+(\psi \cdot Q(u))_t\cdot u_{tt} \Big) 
= \sum_{i=19}^{24} L_i.
\end{split}
\end{equation}

For the terms $L_{19}$--$L_{24}$, we perform the following estimates:
\begin{equation}\label{zhu201xc}
\begin{split}
L_{19}=&-\int   (u\cdot \nabla u)_t \cdot u_{tt}\\
 \leq  & C\big(|u_t|_{2}|\nabla u|_{\infty}+|u|_\infty |\nabla u_t|_2\big)|u_{tt}|_2\leq C( |\phi^e\nabla u_t|^2_2+1)+\frac{1}{10}| u_{tt}|^2_2,\\
L_{20}=&-\int  \nabla \phi_t \cdot u_{tt} 
\leq C|\nabla \phi_t|_2|u_{tt}|_2\leq C|u_{tt}|_2,\\
L_{21}=&-\int a \phi^{2e}_t Lu \cdot u_{tt} 
= a\int (u\cdot \nabla \phi^{2e}+(\delta-1)\phi^{2e}\text{div}u) Lu \cdot u_{tt} \\
\leq & C(|u|_\infty|\psi|_{\infty}|Lu|_{2} +C|\phi^{2e}Lu|_2 |\nabla u|_{\infty})|u_{tt}|_2
\leq C +\frac{1}{10}|u_{tt}|^2_2,\\
\end{split}
\end{equation}
\begin{equation}\label{gugugu}
\begin{split}
L_{22}=&-\int  a \nabla \phi^{2e} \cdot Q(u_t)\cdot u_{tt} \\
\leq & C|\nabla u_t|_2|\psi|_\infty|u_{tt}|_2\leq C |\phi^e\nabla u_t|^2_2+\frac{1}{10}| u_{tt}|^2_2,\\
L_{23}=&\int  \frac{a}{2} \phi^{2e}_t\big(\alpha|\nabla u_t|^2+(\alpha+\beta)|\text{div}u_t|^2\big)\\
=&-\int \frac{a}{2} (u\cdot \nabla \phi^{2e}+(\delta-1)\phi^{2e}\text{div}u)\big(\alpha|\nabla u_t|^2+(\alpha+\beta)|\text{div}u_t|^2\big)\\
\leq& C|u|_\infty|\psi|_\infty |\nabla u_t|^2_2 +C|\nabla u|_\infty|\phi^e\nabla u_t|^2_2,\\
L_{24}=&\int  (\psi \cdot Q(u))_t\cdot u_{tt}
=\int  \psi \cdot Q(u)_t\cdot u_{tt}\text{d}x+\int  \psi_t \cdot Q(u)\cdot u_{tt}\\
\leq& C|\psi|_\infty |\nabla u_t|_2 | u_{tt}|_2-\int (u\cdot \nabla \psi+\delta \psi \text{div}u+a\delta \phi^{2e} \nabla \text{div}u)\cdot Q(u)\cdot u_{tt}\\
\leq &C(|\psi|_\infty |\nabla u_t|_2+|\nabla u|_3(|u|_\infty|\nabla \psi|_6+|\psi|_\infty|\nabla u|_6+|\phi^{2e}\nabla^2 u|_6)) | u_{tt}|_2\\
\leq & C |\phi^e\nabla u_t|^2_2+\frac{1}{10}| u_{tt}|^2_2+C,
\end{split}
\end{equation}
which, along with \eqref{zhu19}-\eqref{zhu201xc}, implies that 
\begin{equation}\label{zhu19x}
\begin{split}
& \frac{1}{2}\frac{d}{dt}\Big(a\alpha|\phi^e\nabla u_t|^2_2+a(\alpha+\beta)|\phi^e\text{div} u_t|^2_2\Big)+| u_{tt}|^2_2
\leq C|\phi^e\nabla u_t|^2_2+C.\end{split}
\end{equation}

Integrating (\ref{zhu19}) over $(\tau,t)$ shows that for $0\leq t \leq T$,
\begin{equation}\label{zhu22g}
\begin{split}
&|\phi^e\nabla u_t(t)|^2_2+\int_{\tau}^t| u_{tt}|^2_2\text{d}s
\leq  C|\phi^e\nabla u_t(\tau)|^2_2+C.
\end{split}
\end{equation}

On the other hand,  it follows from the momentum equations $ (\ref{we22})_2$ that
\begin{equation}\label{zhu15wsx}
\begin{split}
|\phi^e\nabla u_t(\tau)|_2\leq& \big( |\phi^e\nabla \big(u\cdot \nabla u+ \nabla \phi+ a \phi^{2e}  Lu-\psi\cdot Q(u)\big)|_{2}\big)(\tau).
\end{split}
\end{equation}
Then according to the assumptions \eqref{th78}-(\ref{th78zx}) and Lemma \ref{further}, one has
\begin{equation}\label{zhu15vbwsx}
\begin{split}
\lim \sup_{\tau\rightarrow 0}|\phi^e\nabla u_t(\tau)|_2\leq &C( |\phi^e_0\nabla (u_0\cdot \nabla u_0)|_2+|\phi^e_0\nabla^2\phi_0|_2\\
&+|\phi^e_0\nabla (\psi_0\cdot Q(u_0))|_2+| g_3|_{2})\\
\leq& C\big( |\phi^e_0\nabla^2 u_0|_2|u_0|_\infty+|\nabla u_0|_\infty |\phi^e_0\nabla u_0|_2|+|\phi^e_0\nabla^2\phi_0|_2\big)\\
+& C\big( |\phi^e_0\nabla^2u_0|_2|\psi_0|_\infty+|\nabla \psi_0|_3 |\phi^e_0\nabla u_0|_6+1\big)\leq C.
\end{split}
\end{equation}

Actually,  for the estimates of $|\phi^e \nabla^2 \phi|_2$,  due to 
\begin{equation}\label{decom1}
\begin{split}
\phi^e \nabla^2 \phi=&C_3(\rho^{\frac{\delta-1}{2}}\nabla^2 \rho^{\gamma-1})\\
=&C_3\rho^{\frac{\delta-1}{2}} \big((\gamma-1)(\gamma-2)\rho^{\gamma-3}\nabla \rho\otimes \nabla \rho+(\gamma-1) \rho^{\gamma-2}\nabla^2 \rho \big)\\
=&C_4 \rho^{\gamma+\frac{1-3\delta}{2}}\psi\otimes \psi+C_5\rho^{\gamma+\frac{\delta-5}{2}}\nabla^2 \rho,\\
 \nabla^2 \phi^{2e}=&C_6\nabla^2 \rho^{\delta-1}\\
=&C_7 \big((\delta-1)(\delta-2)\rho^{\delta-3}\nabla \rho\otimes \nabla \rho+(\delta-1) \rho^{\delta-2}\nabla^2 \rho \big)\\
=&C_8 \omega \otimes \omega+C_9\rho^{\delta-2}\nabla^2 \rho\in L^2(\mathbb{R}^3),
 \end{split}
\end{equation}
for some constants $C_i$ ($i=3,...,9$),  one can obtain  that 
\begin{equation}\label{decom2}
|\rho^{\delta-2}\nabla^2 \rho|_2+|\rho^{\frac{\delta-1}{2}}\nabla^2 \rho|_2+  |\phi^e \nabla^2 \phi|_2   \leq C \quad \text{for} \quad 0\leq  t \leq T.
\end{equation}

Letting $\tau \rightarrow 0$ in (\ref{zhu22g}), one can obtain that 
\begin{equation}\label{zhu22wsx}
\begin{split}
|\phi^e\nabla u_t|^2_2+ |\nabla u_t|^2_2+\int_{0}^t | u_{tt}|^2_2\text{d}s\leq & C \quad \text{for} \quad  0\leq t \leq T.
\end{split}
\end{equation}

It follows from $(\ref{we22})_2$ that
\begin{equation}\label{zhu77qq}
\begin{split}
a\phi^{2e} Lu_t=&-u_{tt}-(u\cdot\nabla u)_t -\nabla \phi_t-a\phi^{2e} _t Lu+(\psi\cdot Q(u))_t,
\end{split}
\end{equation}
which implies that 
\begin{equation}\label{gaijia1}
\begin{split}
|\phi^{2e} Lu_t(t)|_{2}\leq& C\big| (u_{tt}+(u\cdot\nabla u)_t+\nabla \phi_t-(\psi\cdot Q(u))_t+a\phi^{2e}_t Lu)\big|_2\\
\leq& C\Big(|u_{tt}|_2+|u_t|_6|\nabla u|_3+|u|_\infty|\nabla u_t|_2+|\nabla \phi_t|_2+|\psi|_\infty|\nabla u_t|_2\Big)\\
&+C\Big(|\psi_t|_6|\nabla u|_3+|\psi|_\infty|u|_\infty  |\nabla^2 u|_2+|\phi^{2e}Lu|_6|\text{div}u|_3\Big)\\
\leq & C(1+|u_{tt}|_2).
\end{split}
\end{equation}

Using the  equations $ (\ref{we22})_2$, for multi-index $\xi\in \mathbb{R}^3$ with $|\xi|=2$, one has
\begin{equation}\label{jiegou1}
\begin{split}
|a\phi^e \nabla \big(\phi^{2e}Lu)|_2=&|-a\phi^e \nabla \big(u_t+u\cdot\nabla u +\nabla \phi-\psi \cdot Q(u))|_2\\
\leq & C(|\phi^e\nabla u_t(t)|_2+|\nabla u|_\infty|\phi^e\nabla u|_2+|u|_\infty|\phi^e\nabla^2 u|_2)\\
&+C(|\phi^e \nabla^2 \phi|_2+|\nabla \psi|_3|\phi^e\nabla u|_6+|\psi|_\infty|\phi^e\nabla^2 u|_2)\\
\leq &C(1+|\phi^e \nabla^2 \phi|_2+|\phi^e\nabla^2 u|_2).
\end{split}
\end{equation}

Then according to \eqref{gaijia1},  \eqref{jiegou1}, \eqref{decom2} and \eqref{decom4}, one finally gets
$$
|\phi^e \nabla \big(\phi^{2e}Lu\big)|^2_2 +   \int_{0}^{t}|\phi^{2e}Lu_t|^2_{2}\text{d}s\leq C\quad \text{for} \quad 0\leq  t \leq T.$$

\end{proof}

\begin{lemma}\label{5} 
\begin{equation*}
\begin{split}
\sup_{0\leq t\leq T}\big(t|u_{tt}(t)|^2_2\big)+\int_{0}^{T}t|\phi^e u_{tt}|^2_{D^1}\text{\rm d}t\leq C \quad \text{for} \quad 0\leq  T <\overline{T},\end{split}
\end{equation*}
  where the constant  $C>0$ only depends on $(\rho_0,u_0)$,  $C_0$,  $\alpha$,  $\beta$, $A$, $\gamma$, $\delta$     and $\overline{T}$.
 \end{lemma}

\begin{proof}

Now applying $\partial_t$ to (\ref{zhu37ssh}) yields
\begin{equation}\label{zhu25}
\begin{split}
u_{ttt}+a\phi^{2e}Lu_{tt}
=&-(u \cdot\nabla u)_{tt} -\nabla \phi_{tt}-a\phi^{2e}_{tt}Lu-2a\phi^{2e}_tLu_t\\
&+2\psi_t\cdot Q(u_t)+\psi_{tt}\cdot Q(u)+ \psi\cdot Q(u_{tt}).
\end{split}
\end{equation}
Multiplying  (\ref{zhu25}) by $u_{tt}$ and integrating over $\mathbb{R}^3$ give
\begin{equation}\label{zhu27}
\begin{split}
&\frac{1}{2} \frac{d}{dt}|u_{tt}|^2_2+a\alpha|\phi^{e}\nabla u_{tt}|^2_2+a(\alpha+\beta)|\phi^{e}\text{div} u_{tt}|^2_2\\
=&\int \Big(-(u\cdot\nabla u)_{tt} -\nabla \phi_{tt}-a \nabla \phi^{2e} \cdot Q(u)_{tt}-a\phi^{2e}_{tt}Lu\Big)\cdot u_{tt}\\
&+\int\Big(-2a\phi^{2e}_tLu_t+2\psi_t\cdot Q(u_t)+\psi_{tt}\cdot Q(u)+ \psi\cdot Q(u_{tt}) \Big)\cdot u_{tt}  
= \sum_{i=25}^{32}L_i.
\end{split}
\end{equation}

For the terms $L_{25}$--$L_{32}$, we perform the following estimates:
\begin{equation}\label{zhu201xcxxyy}
\begin{split}
L_{25}=&-\int  (u\cdot\nabla u)_{tt} \cdot u_{tt}\\
\leq & C\big(|\nabla u_t|_6|u_t|_3+ |\nabla u|_\infty | u_{tt}|_2+|u|_\infty| \nabla u_{tt}|_2)| u_{tt}|_2\\
\leq &C (1+|u_{tt}|^2_2+|\nabla^2 u_t|^2_2)+\frac{a\alpha }{10}|\phi^{e}\nabla u_{tt}|^2_2,\\
L_{26}=&-\int  \nabla \phi_{tt}\cdot u_{tt}\\
 \leq & C|\phi_{tt}|_2|\phi^e \nabla u_{tt}|_2|\varphi|^{\frac{1}{2}}_\infty\leq C|\phi_{tt}|^2_2+\frac{a\alpha }{10}|\phi^{e}\nabla u_{tt}|^2_2,\\
L_{27}=&-\int  a \nabla \phi^{2e} \cdot Q(u)_{tt}\cdot u_{tt} \\
\leq &C|\psi|_\infty|u_{tt}|_2  |\phi^e \nabla u_{tt}|_2|\varphi|^{\frac{1}{2}}_\infty\leq C|u_{tt}|^2_2+\frac{a\alpha }{10}|\phi^{e}\nabla u_{tt}|^2_2, \\
L_{28}=&-\int a\phi^{2e}_{tt}Lu\cdot u_{tt}= a\int u\cdot \nabla \phi^{2e}_t  Lu\cdot u_{tt}\\
&-a\int ((\delta-1)\text{div}uu\cdot \nabla \phi^{2e}+(\delta-1)^2\phi^{2e}(\text{div}u)^2)  Lu\cdot u_{tt} \\
&-a\int (u_t\psi+(\delta-1)\phi^{2e}\text{div}u_t)  Lu\cdot u_{tt} \\
\leq& C(|\nabla^2 u|_6|\psi_t|_6|u|_6+|u|_\infty|\nabla u|_\infty|\psi|_\infty|\nabla^2 u|_2+|\phi^{2e}Lu|_2|\text{div}u|^2_\infty)|u_{tt}|_2\\
&+C(|u_t|_6|\psi|_6|Lu|_6|u_{tt}|_2+|\phi^{2e}Lu|_6|\nabla u_t|_2|u_{tt}|_3)\leq C \|u_{tt}\|_1,\\
L_{29}=&-\int  2a\phi^{2e}_tLu_t\cdot u_{tt}\\
=&-\int  2a(u\cdot \nabla \phi^{2e}+(\delta-1)\phi^{2e}\text{div}u)Lu_t\cdot u_{tt}\\
\leq &C(|\psi|_\infty|u|_\infty |Lu_t|_2+|\phi^{2e}Lu_t|_2|\text{div}u|_\infty) | u_{tt}|_2\\
\leq & C (1+|u_{tt}|^2_2+|\nabla^2 u_t|^2_2),\\
L_{30}=&\int  2\psi_t\cdot Q(u_t)\cdot u_{tt}\\
\leq &C|\psi_t|_6 |\nabla u_t|_2 | u_{tt}|_3\leq C (1+|u_{tt}|^2_2)+\frac{a\alpha }{10}|\phi^{e}\nabla u_{tt}|^2_2,\\
L_{31}=&\int  \psi_{tt}\cdot Q(u)\cdot u_{tt}\\
\leq &C(|u_t|_6 |\nabla \psi|_3+|\nabla \psi_t|_2|u|_\infty+|\text{div}u|_3|\psi_t|_6)|\nabla u|_\infty | u_{tt}|_2\\
&+C(|\psi|_\infty\nabla u_t|_2+|\psi|_\infty|u|_\infty|\nabla^2 u|_2)|\nabla u|_\infty | u_{tt}|_2\\
&+C(|\phi^e\nabla u|_3|\nabla^2 u_t|_2|\phi^e u_{tt}|_6+|\phi^e\nabla u|^2_6|\nabla^2 u|_6| u_{tt}|_2)\\
\leq &C (|u_{tt}|^2_2+|\nabla^2 u_t|^2_2)+\frac{a\alpha }{10}|\phi^{e}\nabla u_{tt}|^2_2,\\
L_{32}=&\int  \psi\cdot Q(u_{tt})\cdot u_{tt}\\
\leq  & C|\psi|_\infty |\nabla u_{tt}|_2 | u_{tt}|_2\leq C |u_{tt}|^2_2+\frac{a\alpha }{10}|\phi^{e}\nabla u_{tt}|^2_2,
\end{split}
\end{equation}
which, along with \eqref{gaijia1} and   \eqref{zhu27}, implies that 
\begin{equation}\label{zhu27x}
\begin{split}
&\frac{1}{2} \frac{d}{dt}|u_{tt}|^2_2+a\alpha|\phi^{e}\nabla u_{tt}|^2_2+a(\alpha+\beta)|\phi^{e}\text{div} u_{tt}|^2_2\\
\leq &C (|u_{tt}|^2_2+|\nabla^2 u_t|^2_2+ |\phi_{tt}|^2_2+1).
\end{split}
\end{equation}

Multiplying both sides of (\ref{zhu27x}) by $t$  and  integrating over $(\tau,t)$, one can get
\begin{equation}\label{zhu29}
\begin{split}
&t|u_{tt}(t)|^2_2+\frac{a\alpha}{2}\int_\tau^t s|\phi^e \nabla u_{tt}|^2_2\text{d}s
\leq \tau |u_{tt}(\tau)|^2_2+C.
\end{split}
\end{equation}

It follows  from \eqref{zhu22wsx} and Lemma \ref{1} that
there exists a sequence $s_k$ such that
$$
s_k\rightarrow 0, \quad \text{and}\quad s_k |u_{tt}(s_k,x)|^2_2\rightarrow 0, \quad \text{as} \quad k\rightarrow+\infty.
$$
The desired estimates follows from  taking $\tau=s_k$ and letting $ k\rightarrow+\infty$ in (\ref{zhu29}).

\end{proof}

\subsection{Second order estimates on the velocity with singular weight} Finally, we show that 
  \begin{lemma}\label{s4jjk} 
\begin{equation*}
\begin{split}
\sup_{0\leq t\leq T}|\phi^{2e} \nabla^2 u(t)|_{ 2}+\int_0^T |\phi^{2e}\nabla^2 u_t|^2_2 \text{\rm d}t\leq C \quad \text{for} \quad  0\leq T<  \overline{T}, 
\end{split}
\end{equation*}
  where the constant  $C>0$ only depends on $(\rho_0,u_0)$,  $C_0$,  $\alpha$,  $\beta$, $A$, $\gamma$, $\delta$     and $\overline{T}$.
 \end{lemma}
\begin{proof}
First, it follows from the equations $\eqref{we22}_2$ that  for any constant $\eta>0$, 
\begin{equation}\label{elliptic-yijie1jjk}
\begin{split}
a L((\phi+\eta)^{2e}u)=&a (\phi+\eta)^{2e}Lu-a G(\nabla (\phi+\eta)^{2e},u)=\Lambda_\eta.
\end{split}
\end{equation}
First, due to 
\begin{equation}\label{special11}
\begin{split}
\nabla (\phi+\eta)^{2e}=&\frac{\phi^{-2e+1}}{(\phi+\eta)^{-2e+1}}\nabla \phi^{2e},\\
\nabla^2 (\phi+\eta)^{2e}=&\nabla \Big(\frac{\phi^{-2e+1}}{(\phi+\eta)^{-2e+1}}\nabla \phi^{2e}\Big)\\
=& \frac{\phi^{-2e+1}}{(\phi+\eta)^{-2e+1}}\nabla^2 \phi^{2e}-(2e-1) \frac{2\eta\phi^{-2e+1}}{e(\phi+\eta)^{-2e+2}}\nabla \phi^{e}\otimes \nabla \phi^{e},
\end{split}
\end{equation}
one quickly obtains that 
\begin{equation}\label{special12}
\begin{split}
\|\nabla (\phi+\eta)^{2e}\|_{L^6\cap L^\infty \cap D^1}+|\Lambda_\eta|_2\leq C  \quad \text{for} \quad  0\leq t\leq T, \end{split}
\end{equation}
for some constant $C>0$ that is independent of $\eta$.

Second, due to $\eta>0$ and \eqref{special12}, one has $(\phi+\eta)^{2e}u\in L^\infty([0,T];H^2)$. Then according to Lemma \ref{zhenok} and \eqref{special11}, one has
\begin{equation}\label{special12uyyy}
\begin{split}
| (\phi+\eta)^{2e}u|_{D^2}\leq C|\Lambda_\eta|_2\leq C  \quad \text{for} \quad  0\leq t\leq T, \end{split}
\end{equation}
for some constant $C>0$ that is independent of $\eta$.

Thus one has
\begin{equation}\label{special13}
\begin{split}
|(\phi+\eta)^{2e}\nabla^2 u|_{2}\leq C(1+|\nabla (\phi+\eta)^{2e} |_6|\nabla u|_3+|\nabla^2(\phi+\eta)^{2e}|_2|u|_\infty)\leq C, \end{split}
\end{equation}
for  $0\leq t\leq T$, where the constant $C>0$ is also independent of  $\eta$.

It is very easy to see that, for every $(t,x)\in [0,T]\times \mathbb{R}^3$,
$$
(\phi+\eta)^{2e}\nabla^2 u \rightarrow \phi^{2e}\nabla^2 u \quad \text{a.e.} \quad \text{as} \quad \eta \rightarrow 0,
$$
which, along with Fatou's lemma (see Lemma \ref{Fatou}), implies that 
$$
\int |\phi^{2e}\nabla^2 u|^2 \leq \liminf_{\eta\rightarrow 0} \int  |(\phi+\eta)^{2e}\nabla^2 u|^2 \leq C \quad \text{for} \quad  0\leq t\leq T.
$$

Finally, the estimate on  $\phi^{2e}\nabla^2 u_t$ follows from the following approximate scheme 
$$
a L((\phi+\eta)^{2e}u_t)=a (\phi+\eta)^{2e}Lu_t-a G(\nabla (\phi+\eta)^{2e},u_t),
$$
and the similar argument used above for the estimate on $\phi^{2e}\nabla^2 u$.

 \end{proof}

\subsection{Contradiction argument}\ \\

Now we know from Lemmas \ref{s1}-\ref{s4jjk} that, if the regular solution $(\rho, u)(t,x)$ exists up to the time ${\overline{T}}>0$, with the maximal time ${\overline{T}}<+\infty$ such that  the assumption \eqref{we11} holds,
then 
$$(\rho^{\gamma-1},\nabla \rho^{\delta-1},u)|_{t=\overline{T}} =\lim_{t\rightarrow \overline{T}}(\rho^{\gamma-1},\nabla \rho^{\delta-1},u)$$
 satisfy the conditions imposed on the initial data $(\ref{th78})$. Moreover, if follows from Lemma \ref{verificationcom} and Lemmas \ref{s1}-\ref{s4jjk} that, $$ \sup_{\tau \leq t\leq \overline{T}}\Big(|\phi^e \nabla u(t)|^2_{ 2}+|\phi^{2e}L u(t)|^2_2+|\phi^e \nabla \big(\phi^{2e}Lu(t)\big)|^2_2\Big)\leq C,  $$
for any $\tau\in (0,\overline{T})$, which means that the   compatibility conditions \eqref{th78zx} still hold at $t={\overline{T}}$.
 If we solve the system (\ref{eq:1.1}) with the initial time ${\overline{T}}$, then Theorem 2.1 ensures that $(\rho, u)(t,x)$ extends beyond ${\overline{T}}$ as the unique regular solution. This contradicts to the fact that ${\overline{T}}$ is the maximal existence time. 

Until now,  we have  completed  the proof of Theorem \ref{th3s}.

Finally, we give a detail explanation for  Remark \ref{classicalsense}.
\subsection{The classical solution to the  Cauchy problem   (\ref{eq:1.1})-\eqref{far}}

Now we  show   that, if $\gamma \in (1,2]$,  the regular solution  obtained in the above step is indeed a classical one  in $(0, \overline{T}]\times \mathbb{R}^3$.
First, due to  $1<\gamma \leq 2$, one has
$$(\rho, \nabla \rho,  \rho_t,  u,\nabla u)\in C( [0,\overline{T}]\times \mathbb{R}^3).$$

Second, it follows from the classical Sobolev embedding  theorem:
\begin{equation}\label{qian}\begin{split}
&L^2([0,\overline{T}];H^1)\cap W^{1,2}([0,\overline{T}];H^{-1})\hookrightarrow C([0,\overline{T}];L^2),
\end{split}
\end{equation}
and the regularity (\ref{reg11}) that for any $\tau\in (0,\overline{T})$,
\begin{equation}\label{qiaf}
 tu_t\in C([0,\overline{T}]; H^2), \quad \text{and} \quad u_t\in C([\tau,\overline{T}]\times \mathbb{R}^3).
\end{equation}

Finally, it remains to  show that
 $ Lu \in C([\tau,\overline{T}]\times \mathbb{R}^3)$. According to  the following elliptic system
\begin{equation}\label{ekki}
\begin{split}
aLu=-\phi^{-2e}\big(u_t+u\cdot\nabla u +\nabla \phi-\psi \cdot Q(u))=\phi^{-2e}H,
\end{split}
\end{equation}
the regularity (\ref{reg11}) implies that
\begin{equation}\label{er1}
\begin{split}
t \phi^{-2e} H \in L^\infty([0,\overline{T}]; H^2).
\end{split}
\end{equation}
Note that
\begin{equation} \label{chejj}
\begin{split}
(t\phi^{-2e} H)_t=&\phi^{-2e} H+t\phi^{-2e}_t H +t\phi^{-2e}H_t \in L^2([0,\overline{T}]; L^2).
\end{split}
\end{equation}
So it follows from the   classical  Sobolev imbedding theorem:
 \begin{equation}\label{yhnh}
\begin{split}
L^\infty([0,\overline{T}];H^1)\cap W^{1,2}([0,\overline{T}];H^{-1})\hookrightarrow C([0,\overline{T}];L^q),
\end{split}
\end{equation}
for any $q\in [2,6)$, and (\ref{ekki})-(\ref{chejj}) that
$$
t\phi^{-2e} H\in C([0,\overline{T}]; W^{1,4}), \quad
tL u \in C([0,\overline{T}]; W^{1,4}).
$$
Again  the  Sobolev embedding theorem implies that
$
L u \in C((0,\overline{T}]\times \mathbb{R}^3)
$.

\section{Periodic problem away from the vacuum}

This section will be devoted to proving   Theorem \ref{th3t}, and we always assume that \eqref{canshu} and \eqref{bdrelation} hold, and $\Omega=\mathbb{T}^3$.   
When $\inf_x \rho_0(x) > 0$, the local existence  of the unique  classical solution to  the periodic problem \eqref{eq:1.1}-\eqref{initial} stated  in Theorem 1.3 follows from the standard theory of the  symmetric hyperbolic-parabolic structure which satisfies the well-known Kawashima’s condition, c.f. \cite{KA, nash}.

   In order to prove (\ref{crit}), we use a contradiction argument. 
Let $( \rho, u)$ be the unique classical  solution  stated in Theorem 1.3 to  the periodic  problem  (\ref{eq:1.1})-\eqref{initial} with the life span $\overline{T}$ .
 We assume that
$\overline{T}<+\infty$ and
\begin{equation}\label{we11ty}
\begin{split}
\displaystyle\lim_{T\mapsto \overline{T}} \int_0^T \|D( u)(t,\cdot)\|^2_ {L^\infty(\mathbb{T}^3)}\ \text{d}t=C_0<\infty,
\end{split}
\end{equation}
for some constant $C_0>0$.
We will show that under the assumption \eqref{we11ty}, $\overline{T}$ is actually not the maximal existence time for this classical  solution.

First, it  follows from the  assumption (\ref{we11ty}) and the system (\ref{eq:1.1}) that  $\rho$ is uniformly bounded.
 \begin{lemma}\label{bs1}
\begin{equation*}
\begin{split}
C^{-1}\leq \rho(t,x) \leq C \quad \text{for} \quad 0\leq t< \overline{T},
\end{split}
\end{equation*}
where the constant $C>0$ only depends on $(\rho_0,u_0)$,  $C_0$   and $\overline{T}$.
 \end{lemma}
 The proof is similar to that of Lemma \ref{s1}. Here we omit it. 
 
 Second, we give the well-known B-D entropy estimate.
  \begin{lemma}\cite{bd}\label{bskk2}
\begin{equation*}
\begin{split}
\int \Big(\frac{1}{2}\rho\Big|u+\frac{2\alpha \delta}{\delta-1}\nabla \rho^{\delta-1}\Big|^2+\frac{P}{\gamma-1}\Big)(t,\cdot)&\\
+\int_0^t \int_{\mathbb{T}^3} \Big(\frac{2P'(\rho)\mu'(\rho)}{\rho}|\nabla \rho|^2+\frac{1}{2}\mu(\rho)|\nabla u-\nabla u^\top|^2\Big)\text{\rm d}x\text{\rm d}s &\leq C,
\end{split}
\end{equation*}
for $0\leq t< \overline{T}$, where the constant $C>0$ only depends on $(\rho_0,u_0)$, $A$,  $\gamma$, $\alpha$   and $\delta$.
 \end{lemma} 

The proof of this lemma can be found in \cite{bd}. Here we omit it. Now we give the basic energy estimate.
  \begin{lemma}\cite{bd}\label{bskk3}
\begin{equation*}
\begin{split}
\int \Big(\frac{1}{2}\rho|u|^2+\frac{P}{\gamma-1}\Big)(t,\cdot)+\int_0^t \int_{\mathbb{T}^3}  \Big(\frac{1}{2}\mu(\rho)|\nabla u+\nabla u^\top|^2+\lambda(\rho)(\text{div}u)^2\Big)\text{\rm d}x\text{\rm d}s &\leq C,
\end{split}
\end{equation*}
for $0\leq t< \overline{T}$, where the constant $C>0$ only depends on $(\rho_0,u_0)$, $A$ and   $\gamma$.
 \end{lemma}  

The conclusion obtained above is classical. Here we omit its proof.

 It follows quickly from Lemmas \ref{bs1}-\ref{bskk3} that 
  \begin{lemma}\label{ts2}    
\begin{equation*}
\begin{split}
\sup_{0\leq t\leq T}\big(|u|^2_{ 2}+|\nabla \rho|^2_2\big)(t)+\int_{0}^{T}|\nabla u|^2_{2}\text{\rm d}t\leq C\quad \text{for} \quad 0\leq T<  \overline{T},
\end{split}
\end{equation*}
where the constant $C>0$ only depends on $(\rho_0,u_0)$,  $C_0$,  $\alpha$,  $\beta$, $A$, $\gamma$, $\delta$     and $\overline{T}$.
 \end{lemma}

Now we improve the energy estimate obtained in Lemma \ref{ts2}.
  \begin{lemma}\label{abs3}
\begin{equation}\label{keyq}
\begin{split}
&\frac{d}{dt}\int \varphi |u|^4(t,\cdot)+ \int  |u|^2 |\nabla u|^2(t,\cdot)\\
\leq & C\big(1+|\text{div}u|_\infty|\sqrt{\varphi}|u|^2|^2_2+|D(u)|_\infty||u|^3|_2\big)\quad  \text{for} \quad 0\leq t< \overline{T},
\end{split}
\end{equation}
where the constant $C>0$ only depends on $(\rho_0,u_0)$,  $C_0$,  $\alpha$,  $\beta$, $A$, $\gamma$, $\delta$     and $\overline{T}$.

 \end{lemma}

\begin{proof}

 First, multiplying $ (\ref{dege})_3$ by $r|u|^{r-2}u$ $(r\geq 3)$ and integrating the resulting equation over $\mathbb{T}^3$ by parts, then one has
\begin{equation}\label{lz1}
\begin{split}
 \frac{d}{dt}\int & \varphi |u|^r+a\int H_r
=-ar(r-2)(\alpha+\beta)\int  \text{div}u |u|^{r-3}u\cdot \nabla |u|\\
&+\int  \Big(\delta \varphi \text{div}u |u|^r+r g \text{div }(|u|^{r-2}u)+rf\cdot Q(u)\cdot |u|^{r-2}u\Big),
\end{split}
\end{equation}
where
$$
H_r=r|u|^{r-2}\big(\alpha |\nabla u|^2+(\alpha+\beta)|\text{div}u|^2+\alpha(r-2)|\nabla |u||^2\big).
$$

For any given $\epsilon_1\in (0,1)$, we define a nonnegative function which will be determined in $\textbf{Step}\ 2$ as follows
$$
\Phi(\epsilon_0,\epsilon_1,r)=\left\{ \begin{array}{llll}
\frac{\alpha \epsilon_1(r-1)}{3\big(-\frac{\alpha(4-\epsilon_0)}{3}-\beta+\frac{r^2(\alpha+\beta)}{4(r-1)}\big)}, \quad \text{if}\quad \frac{r^2(\alpha+\beta)}{4(r-1)} -\frac{\alpha(4-\epsilon_0)}{3}-\beta>0, \\[12pt]
\displaystyle
0,\quad \text{otherwise}. \end{array}\right.
$$

$\textbf{Step}\ 1$: We assume that 
\begin{equation}\label{ghu}
\begin{split}
&\int_{\mathbb{T}^3 \cap \{|u|>0\}}  |u|^r \Big| \nabla \Big(\frac{u}{|u|}\Big)\Big|^2\text{d}x> \Phi(\epsilon_0,\epsilon_1,r)\int_{\mathbb{T}^3 \cap \{|u|>0\}}  |u|^{r-2} \big| \nabla  |u|\big|^2\text{d}x.
\end{split}
\end{equation}
 A direct calculation gives for $|u|>0$ that  the following formula holds.
 \begin{equation}\label{gpkk}
\displaystyle
|\nabla u|^2=|u|^2\Big| \nabla \Big(\frac{u}{|u|}\Big)\Big|^2+\big| \nabla  |u|\big|^2.
\end{equation}

According to  (\ref{lz1}) and  the Cauchy's inequality, one has
\begin{equation}\label{lz3}
\begin{split}
& \frac{d}{dt}\int \varphi |u|^r+a\int_{\mathbb{T}^3 \cap \{|u|>0\}} H_r\text{d}x\\
=&-ar(r-2)(\alpha+\beta)\int_{\mathbb{T}^3 \cap \{|u|>0\}}  \text{div}u |u|^{\frac{r-2}{2}} |u|^{\frac{r-4}{2}}u\cdot \nabla |u|\text{d}x\\
&+\int  \Big(\delta \varphi \text{div}u |u|^r+r g \text{div }(|u|^{r-2}u)+rf\cdot Q(u)\cdot |u|^{r-2}u\Big)\\
\leq & ar(\alpha+\beta)\int_{\mathbb{T}^3 \cap \{|u|>0\}}   |u|^{r-2} \Big( |\text{div}u|^2
+\frac{(r-2)^2}{4} |\nabla |u||^2\Big)\text{d}x\\
&+\int  \Big(\delta \varphi \text{div}u |u|^r+r g \text{div }(|u|^{r-2}u)+rf\cdot Q(u)\cdot |u|^{r-2}u\Big).
\end{split}
\end{equation}
It follows from Lemmas \ref{bs1}- \ref{ts2},   Gagliardo-Nirenberg inequality,    H\"older's inequality and Young's inequality that
\begin{equation}\label{zhu2s}
\begin{split}
M_1=&\int  \delta \varphi \text{div}u |u|^r \leq C |\text{div}u|_\infty |\sqrt{\varphi}|u|^{\frac{r}{2}}|^2_2,\\
M_2=& r\int g|u|^{r-2}|\nabla u|  
\leq C \Big( \int |u|^{r-2}|\nabla u|^2 \Big)^{\frac{1}{2}}\Big(\int |u|^{r-2}g^2\Big)^{\frac{1}{2}}\\
\leq&C\Big( \int |u|^{r-2}|\nabla u|^2 \Big)^{\frac{1}{2}}||u|^{\frac{r}{2}}|^{1-\frac{2}{r}}_6|g|_{\frac{6r}{2r+2}}\\
\leq &
\frac{1}{4}a\alpha r \epsilon_0 \int |u|^{r-2}|\nabla u|^2\text{d}x+C(a, \alpha, r,\epsilon_0),\\
M_3=&\int rf\cdot Q(u)\cdot |u|^{r-2}u\leq C|\nabla \rho|_2|D(u)|_\infty||u|^{r-1}|_2,
\end{split}
\end{equation}
where $\epsilon_0\in (0,\frac{1}{4})$  is independent of $r$.
Then combining (\ref{ghu})-(\ref{zhu2s}), one  easily has 
\begin{equation}\label{lz4}
\begin{split}
& \frac{d}{dt}\int  \varphi |u|^r +ar\Psi(\epsilon_0,\epsilon_1,\epsilon_2,r)\int_{\mathbb{T}^3 \cap \{|u|>0\}} |u|^{r-2}|\nabla |u||^2\text{d}x\\
&+\int_{\mathbb{T}^3 \cap \{|u|>0\}}  a\alpha r(1-\epsilon_0)\epsilon_2|u|^{r}\Big| \nabla \Big(\frac{u}{|u|}\Big)\Big|^2\text{d}x\\
\leq &C(a,\alpha, r,\epsilon_0)+C |\text{div}u|_\infty |\sqrt{\varphi}|u|^{\frac{r}{2}}|^2_2+C|D(u)|_\infty||u|^{r-1}|_2,
\end{split}
\end{equation}
where
\begin{equation*}
\begin{split}
\Psi(\epsilon_0,\epsilon_1,\epsilon_2, r)=\alpha (1-\epsilon_0)(1-\epsilon_2)\Phi(\epsilon_0,\epsilon_1,r)+\alpha(r-1-\epsilon_0)-\frac{(r-2)^2(\alpha+\beta)}{4}.
\end{split}
\end{equation*}

Via choosing $\epsilon_0<2(1-\delta)$ small enough, then  $\beta <-\epsilon_0 \alpha$, i.e., 

 $$4\notin \left\{ r\big| \frac{r^2(\alpha+\beta)}{4(r-1)}-\frac{(4-\epsilon_0)\alpha}{3}-\beta>0\right\}.$$ 
  In this case, for $r=4$ and $0<\epsilon_0 <\min\{2(1-\delta),\frac{1}{4}\}$,  it is easy to get
 \begin{equation}\label{peng}
\begin{split}
&r\Big[\alpha (1-\epsilon_0)(1-\epsilon_2)\Phi(\epsilon_0,\epsilon_1,r)+\alpha(r-1-\epsilon_0)-\frac{(r-2)^2(\alpha+\beta)}{4}\Big]\\
>&4\Big(\frac{11}{4}\alpha-(\alpha+\beta)\Big)=4\Big(\frac{7\alpha}{4}-\beta\Big)
\geq 4\Big(\frac{7\alpha}{4}+\epsilon_0\alpha\Big)>7\alpha,
\end{split}
\end{equation}
which, together with (\ref{lz4}),  implies that
\begin{equation}\label{lz422}
\begin{split}& \frac{d}{dt}\int \varphi  |u|^4+C\int_{\mathbb{T}^3 \cap \{|u|>0\}}|u|^{2}|\nabla u|^2\text{d}x\\
\leq &C(a, \alpha, r,\epsilon_0)+C |\text{div}u|_\infty |\sqrt{\varphi}|u|^{2}|^2_2+C|D(u)|_\infty||u|^{3}|_2,
\end{split}
\end{equation}

$\textbf{Step}$ 2 : we assume that
\begin{equation}\label{ghu11}
\begin{split}
&\int_{\mathbb{T}^3 \cap \{|u|>0\}}  |u|^r \Big| \nabla \Big(\frac{u}{|u|}\Big)\Big|^2\text{d}x\leq  \Phi(\epsilon_0,\epsilon_1,r)\int_{\mathbb{T}^3 \cap \{|u|>0\}}  |u|^{r-2} \big| \nabla  |u|\big|^2\text{d}x.
\end{split}
\end{equation}
A direct calculation gives for $|u|>0$,
\begin{equation}\label{ghu22}
\begin{split}
\text{div}u=|u|\text{div}\Big(\frac{u}{|u|}\Big)+\frac{u\cdot \nabla |u|}{|u|}.
\end{split}
\end{equation}
Then combining (\ref{ghu22}) and (\ref{lz3})-(\ref{zhu2s}), one  quickly has
\begin{equation}\label{lz77}
\begin{split}
& \frac{d}{dt}\int \varphi |u|^r +a\int_{\mathbb{T}^3 \cap \{|u|>0\}}\alpha r(1-\epsilon_0)|u|^{r-2}|\nabla u|^2\text{d}x\\
&+a\int_{\mathbb{T}^3 \cap \{|u|>0\}}\Big(r(\alpha+\beta) |u|^{r-2}|\text{div}u|^2+\alpha r(r-2)|u|^{r-2}\big| \nabla  |u| \big|^2\Big)\text{d}x\\
=&-ar(r-2)(\alpha+\beta)\int_{\mathbb{T}^3 \cap \{|u|>0\}}  |u|^{r-2} u\cdot \nabla |u| \text{div}\Big(\frac{u}{|u|}\Big)\text{d}x\\
&-ar(r-2)(\alpha+\beta)\int_{\mathbb{T}^3 \cap \{|u|>0\}} |u|^{r-4} |u\cdot \nabla |u| |^2\text{d}x\\
&+\int  \Big(\delta \varphi \text{div}u |u|^r+rf\cdot Q(u)\cdot |u|^{r-2}u\Big).
\end{split}
\end{equation}
This gives
\begin{equation}\label{lz88}
\begin{split}
& \frac{d}{dt}\int\varphi |u|^r +\int_{\mathbb{T}^3 \cap \{|u|>0\}}a r|u|^{r-4}\Gamma \text{d}x\\
\leq &C(a, \alpha, r,\epsilon_0)+C |\text{div}u|_\infty |\sqrt{\varphi}|u|^{\frac{r}{2}}|^2_2+C|D(u)|_\infty||u|^{r-1}|_2,
\end{split}
\end{equation}
where
\begin{equation}\label{wang1}
\begin{split}
\Gamma=&\alpha(1-\epsilon_0)|u|^{2} |\nabla u|^2+(\alpha+\beta)|u|^{2}|\text{div}u|^2+\alpha(r-2)|u|^{2}\big| \nabla  |u| \big|^2\\
&+(r-2)(\alpha+\beta)|u|^{2} u\cdot \nabla |u|\text{div}\Big(\frac{u}{|u|}\Big)+(r-2)(\alpha+\beta)|u \cdot \nabla |u||^2.
\end{split}
\end{equation}
Now we consider how to make sure that $\Gamma \geq 0$.
\begin{equation*}
\begin{split}
\Gamma=&\alpha (1-\epsilon_0)|u|^{2} \Big( |u|^2\Big| \nabla \Big(\frac{u}{|u|}\Big)\Big|^2+\big| \nabla  |u|\big|^2\Big)\\
&+(\alpha+\beta)|u|^{2}|\Big(|u|\text{div}\Big(\frac{u}{|u|}\Big)+\frac{u\cdot \nabla |u|}{|u|}\Big)^2+\alpha(r-2)|u|^{2}\big| \nabla  |u| \big|^2\\
&+(r-2)(\alpha+\beta)|u|^{2} u\cdot \nabla |u|\text{div}\Big(\frac{u}{|u|}\Big)+(r-2)(\alpha+\beta)|u \cdot \nabla |u||^2\\
=&\alpha(1-\epsilon_0) |u|^{4} \Big| \nabla \Big(\frac{u}{|u|}\Big)\Big|^2+\alpha(r-1-\epsilon_0)|u|^{2}\big| \nabla  |u| \big|^2\\
\end{split}
\end{equation*}
\begin{equation*}
\begin{split}
&+(r-1)(\alpha+\beta)\Big(u \cdot \nabla |u|+\frac{r}{2(r-1)}|u|^{2}\Big(\text{div}\frac{u}{|u|}\Big)\Big)^2\\
&+(\alpha+\beta)|u|^{4}\Big(\text{div}\frac{u}{|u|}\Big)^2-\frac{r^2(\alpha+\beta)}{4(r-1)}|u|^{4}\Big(\text{div}\Big(\frac{u}{|u|}\Big)\Big)^2,
\end{split}
\end{equation*}
which, combining with the fact
$
\Big|\text{div}\Big(\frac{u}{|u|}\Big)\Big|^2\leq 3\Big|\nabla \Big(\frac{u}{|u|}\Big)\Big|^2
$,
implies that
\begin{equation*}
\begin{split}
\Gamma
\geq &\alpha(r-1-\epsilon_0)|u|^{2}\big| \nabla  |u| \big|^2+\Big(\frac{(4-\epsilon_0)\alpha}{3}+\beta-\frac{r^2(\alpha+\beta)}{4(r-1)}\Big)|u|^{4}\Big(\text{div}\Big(\frac{u}{|u|}\Big)\Big)^2.
\end{split}
\end{equation*}
Thus
\begin{equation}\label{wang4}
\begin{split}
&\int_{\mathbb{T}^3 \cap \{|u|>0\}} r|u|^{r-4} \Gamma\text{d}x\geq \alpha r(r-1-\epsilon_0)\int_{\mathbb{T}^3 \cap \{|u|>0\}} |u|^{r-2} \big| \nabla  |u| \big|^2\text{d}x\\
&\qquad + r \Big(\frac{(4-\epsilon_0)\alpha}{3}+\beta-\frac{r^2(\alpha+\beta)}{4(r-1)}\Big) \int_{\mathbb{T}^3 \cap \{|u|>0\}} |u|^{r} \Big(\text{div}\Big(\frac{u}{|u|}\Big)\Big)^2\text{d}x\\
&\qquad \geq  \Lambda(\epsilon_0,\epsilon_1,r)\int_{\mathbb{T}^3 \cap \{|u|>0\}}  |u|^{r-2} \big| \nabla  |u| \big|^2\text{d}x,
\end{split}
\end{equation}
where
\begin{equation}\label{xuchendd}
\begin{split}
\Lambda(\epsilon_0,\epsilon_1,r)=3r \Big(\frac{(4-\epsilon_0)\alpha}{3}+\beta-\frac{r^2(\alpha+\beta)}{4(r-1)}\Big)\Phi(\epsilon_0,\epsilon_1,r)+\alpha r(r-1-\epsilon_0).
\end{split}
\end{equation}
Here we need that $\epsilon_0$ is sufficiently small such that 
$$0<\epsilon_0 <\min\{2(1-\delta),1/4, (r-1)(1-\epsilon_1)\}.$$
Then combining (\ref{gpkk}),  (\ref{lz88}) and  (\ref{wang4})-(\ref{xuchendd}), when $r=4$, one  quickly has 
\begin{equation}\label{lz4mm}
\begin{split}
&\frac{d}{dt}\int \varphi |u|^4(t,\cdot)+ \int  |u|^2 |\nabla u|^2(t,\cdot)\\
\leq & C\big(1+|\text{div}u|_\infty|\sqrt{\varphi}|u|^2|^2_2+|D(u)|_\infty||u|^3|_2\big)\quad  \text{for} \quad 0\leq t< \overline{T}.
\end{split}
\end{equation}

Finally,  \eqref{keyq} follows from \eqref{lz422} and \eqref{lz4mm}.

\end{proof}

The next lemma provides a  key estimate on $\nabla u $.

  \begin{lemma}\label{s4ty}  
\begin{equation*}
\begin{split}
\sup_{0\leq t\leq T}\big(|\nabla u|^2_{ 2}+|u|^4_{ 4}\big)(t)+\int_0^T (|\nabla^2 u|^2_2+| u_t|^2_2+||u||\nabla u||^2)\text{\rm d}t\leq C,
\end{split}
\end{equation*}
for $ 0\leq T<  \overline{T}$,   where the constant  $C>0$ only depends on $(\rho_0,u_0)$,  $C_0$,  $\alpha$,  $\beta$, $A$, $\gamma$, $\delta$     and $\overline{T}$.
 \end{lemma}
\begin{proof}
First, it follows from the proof of Lemma \ref{s4} that 
\begin{equation}\label{zhu6ty}
\begin{split}
&\frac{a}{2} \frac{d}{dt}\Big(\alpha|\phi^e\nabla u|^2_2+(\alpha+\beta)|\phi^e\text{div}u|^2_2\Big)+\int (a\phi^{2e}Lu+ \nabla \phi)^2\equiv: \sum_{i=1}^{8} L_i.
\end{split}
\end{equation}
The definitions of $L_i$ $(i=1,..,8)$ are same as those of $L_i$ $(i=1,..,8)$ in Lemma \ref{s4} with $\mathbb{R}^3$ replaced by $\mathbb{T}^3$.

  Second, for the terms $L_1$, $L_3$ and $L_4$,  similarly one still has 
\begin{equation}\label{zhu10tty}
\begin{split}
|L_1|
 \leq & C|D(u)|_\infty|\phi^{e}\nabla u|^2_2+\epsilon |\nabla^2 u|^2_2+C(\epsilon)|\phi^e \nabla u |^4_2,\\
  |L_3|
\leq &C|\phi|^{1-2e}_\infty|\phi^e\nabla u|^2_2+C|\text{div}u|_{\infty}|u|_2|\nabla\phi|_2,\\ 
L_{4}
\leq &\frac{d}{dt} \int  \phi \text{div}u +C|\text{div}u|_{\infty}|u|_2|\nabla\phi|_2+C|\phi|^{1-2e}_\infty|\phi^e\nabla u|^2_2.
 \end{split}
\end{equation} 

For other terms, according to      H\"older's inequality,  Gagliardo-Nirenberg inequality  and Young's inequality, one gets
\begin{equation}\label{zhu10t3355ty}
\begin{split}
|L_2|
=& a(2\alpha+\beta)\Big|\int \phi^{2e} \big(-\nabla u: \nabla u^\top  \text{div}u+\frac{1}{2}(\text{div}u)^3\big)\Big|\\
&+ a(2\alpha+\beta)\Big|\int (u\cdot \nabla) u\cdot \nabla \phi^{2e}\text{div}u\Big|\\
\leq & C|\text{div} u|_\infty |\phi^{e}\nabla u|^2_2+C|\psi|_2|u|_3|\nabla u |_6|\text{div} u|_\infty\\
 \leq & C(1+|D(u)|^2_\infty)(1+|\phi^{e}\nabla u|^2_2)+\epsilon |\nabla^2 u|^2_2, \\ 
L_{5}=&\frac{1-\delta}{\delta}\int \psi\cdot Q(u)\cdot u_t
\leq  C|D(u)|_\infty|\psi|_2|u_t|_2,\\
L_{6}
=& \frac{a}{2}\int (-\text{div}(u \phi^{2e})-(\delta-2)\phi^{2e} \text{div}u )(\alpha |\nabla u|^2+(\alpha+\beta)|\text{div} u|^2)\\
\leq & C||u||\nabla u||_2|\nabla^2 u |_2|\phi^{2e}|_\infty+ C|D(u)|_\infty|\phi^{e}\nabla u|^2_2 \\
 \leq & C|D(u)|_\infty|\phi^{e}\nabla u|^2_2+\epsilon |\nabla^2 u|^2_2+C(\epsilon)||u||\nabla u||^2_2, \\
L_{7}=&\int \psi\cdot Q(u) \cdot  \nabla \phi 
\leq  C|D(u)|_\infty|\psi|_2|\nabla \phi|_2,\\
L_{8}=&\int \psi\cdot Q(u) \cdot \phi^{2e}Lu
\leq  C|D(u)|_\infty|\psi|_2|\phi^{2e}Lu|_2\\
\leq & C(\epsilon)|D(u)|^2_\infty+\epsilon |\phi^{2e}Lu|^2_2,\end{split}
\end{equation}
where $\epsilon> 0$ is a sufficiently small constant.

It follows from  (\ref{zhu6ty})-(\ref{zhu10t3355ty})  that 
\begin{equation}\label{zhu6qssty}
\begin{split}
& \frac{d}{dt}\int  \Big(\frac{a}{2}\alpha|\phi^e\nabla u|^2+\frac{a}{2}(\alpha+\beta)|\phi^e\text{div}u|^2-\phi \text{div}u\Big)\text{d}x\\
&+C|\nabla^2 u|^2_2+\frac{a}{2}|\phi^{2e} Lu|^2_2\\
\leq & C(1+|D(u)|^2_\infty+|\phi^{e}\nabla u|^2_2)(1+|\phi^{e}\nabla u|^2_2)+C(\epsilon) ||u||\nabla u||^2_2.\end{split}
\end{equation}

Let $\eta>0$ be a sufficiently small constant. We add $\eta$\eqref{zhu6qssty} to \eqref{keyq}, and it follows from Gronwall's inequality that 
\begin{equation*}
\begin{split}
|\phi^e \nabla u|^2_{ 2}+|\sqrt{\varphi}|u|^2|^2_2+\int_0^t (|\nabla^2 u|^2_2+|\phi^{2e} Lu|^2_2+||u||\nabla u||^2_2)   \text{d}s\leq C,
\end{split}
\end{equation*}
for $ 0\leq t\leq T$, which, together with \eqref{ut}, implies that 
\begin{equation*}
\begin{split}
\int_0^t | u_t|^2_2\text{d}s\leq C\int_0^t (|\phi^{2e} L u|^2_2+|\nabla u|^2_3|u|^2_6+|\nabla \phi|^2_2+|D( u)|^2_\infty|\psi|^2_2)\text{d}s\leq C.
\end{split}
\end{equation*}
 \end{proof}

\begin{lemma}\label{s6hty}
 \begin{equation}\label{zhu14sshty}
\begin{split}
&\sup_{0\leq t\leq T}\Big(|u_t|^2_2+|u|^2_{D^2}+|\nabla \rho|^2_6\Big)(t)+\int_0^T(| \nabla u_t|^2_2+ |u|^2_{D^{2,6}})   \text{\rm d}t\leq C
\end{split}
\end{equation}
for $0\leq T<  \overline{T}$,
where the constant  $C>0$ only depends on $(\rho_0,u_0)$,  $C_0$,  $\alpha$,  $\beta$, $A$, $\gamma$, $\delta$     and $\overline{T}$.
\end{lemma}
\begin{proof}
First, it  follows from Lemma \ref{zhenok}, equations $\eqref{dege}_3$ and Young's inequality    that 
$$
|u|_{D^2}\leq  C\big(|u_t|_{2}+| u|^2_6|\nabla u|_2+|\nabla \phi|_2+|\psi \cdot Q(u)|_2\big),
$$
which, along with  Lemmas \ref{bs1}-\ref{s4ty}, implies that 
 \begin{equation}\label{zhu15nnhty}
\begin{split}
|u|_{D^2}\leq C(1+|u_t|_2+|\nabla \rho|^2_6)\quad \text{or} \quad |u|_{D^2}\leq C(1+|u_t|_2+|D(u)|_\infty).
\end{split}
\end{equation}

Second,  differentiating $(\ref{eq:1.1})_2$ with respect to t, it reads
\begin{equation}\label{wgh78}
\begin{split}
&\rho u_{tt}-\text{div}(2\mu(\rho)D(u_t)+\lambda(\rho)\text{div}u_t\,\mathbb{I}_3)\\
=&-\rho_tu_t -\rho u\cdot\nabla u_t-\rho_t u\cdot\nabla u-\rho u_t\cdot\nabla u-\nabla P_t\\
&+\text{div}(2\mu(\rho)_tD(u)+\lambda(\rho)_t\text{div}u\,\mathbb{I}_3).
\end{split}
\end{equation}
Multiplying (\ref{wgh78}) by $u_t$ and   integrating  over $\mathbb{T}^3$, one has 
\begin{equation}\label{wzhen4}
\begin{split}
&\frac{1}{2}\frac{d}{dt}\int \rho |u_t|^2 +\int \Big(\frac{1}{2}\mu(\rho)|\nabla u_t+\nabla u^\top_t|^2+\lambda(\rho)(\text{div}u_t)^2\Big) \\
=&  -\int \rho u \cdot \nabla  |u_t|^2-\int  \rho u \nabla ( u \cdot \nabla u \cdot u_t)-\int  \rho u_t \cdot \nabla u \cdot u_t\\
&+\int  P_t \text{div} u_t-\int \Big(2\mu(\rho)D(u):D(u_t)+\lambda(\rho)_t\text{div}u\text{div}u_t \Big)
\equiv:\sum_{i=9}^{13}\tilde{L}_i.
\end{split}
\end{equation}

According to Lemmas \ref{bs1}-\ref{s4ty}, H\"older's inequality, Gagliardo-Nirenberg inequality and Young's inequality,  we deduce that
\begin{equation}\label{wzhou6opp}
\begin{split}
\tilde{L}_{9}=&-\int \rho u \cdot \nabla  |u_t|^2\\
\leq &  C|\rho|_{\infty}|u|_{\infty}|u_t|_{2}|\nabla u_t|_{2}\leq C\|\nabla u\|^2_1|u_t|^2_{2}+\epsilon |\nabla  u_t|^2_2, \\
\tilde{L}_{10}=& -\int \rho u\cdot  \nabla ( u \cdot \nabla u \cdot u_t)\\
\leq& C\big( |u_t|_6 ||\nabla u|^2|_{\frac{3}{2}} |u|_{6}+ ||u|^2|_{3} |\nabla^2 u|_{2} |u_t|_{6}+||u|^2|_{3} |\nabla u|_{6}  | \nabla u_t|_{2}\big)\\
\leq& C \big( |\nabla u|^2_{3} |\nabla u|_{2}+ |\nabla u|^2_{2} \|\nabla u\|_{1}   \big)|\nabla u_t|_{2}\\
\leq & C\|\nabla u\|_{1}| \nabla u_t|_{2} \leq \epsilon|\nabla u_t|^2_{2}+C(\epsilon)\|\nabla u\|^2_{1},\\
\tilde{L}_{11}=&-\int \rho u_t \cdot \nabla u \cdot u_t \\
\leq &  C|\rho|_{\infty}|u_t|_{6}| u_t|_{2}|\nabla u|_{3}
\leq \epsilon|\nabla u_t|^2_{2}+C(\epsilon)| u_t|^2_{2} \|\nabla u\|^2_{1},\\
\end{split}
\end{equation}
\begin{equation}\label{wzhou6}
\begin{split}
\tilde{L}_{12}=&\int P_t \text{div} u_t \leq \int_{\Omega}|u\cdot \nabla P+\gamma P\text{div}v| |\nabla u_t| \\
\leq &C(|u|_{\infty}|\nabla P|_{2}+|P|_{\infty}|\text{div} u|_{2})|\nabla u_t|_{2}\\
\leq & \epsilon|\nabla u_t|^2_{2}+C(\epsilon)\|\nabla u\|^2_{1},\\
\tilde{L}_{13}=&\int \Big(2\mu(\rho)D(u):D(u_t)+\lambda(\rho)_t\text{div}u\text{div}u_t \Big) \\
\leq &  C|D(u)|_{\infty}|\nabla u_t|_{2}(|\rho|_\infty |\nabla u|_{2}+| u|_{3}|\nabla \rho|_6)\\
\leq& \epsilon|\nabla u_t|^2_{2}+C(\epsilon)|D(u)|^2_{\infty}(1+|\nabla \rho|_6).
\end{split}
\end{equation}
Then it follows from   \eqref{wzhen4}-\eqref{wzhou6} and Lemmas \ref{s4ty} and \ref{Korn} that 
\begin{equation}\label{wzhen5g}
\begin{split}
&\frac{1}{2}\frac{d}{dt}\int \rho |u_t|^2 +\int |\nabla u_t|^2\\
\leq & C(|\nabla \rho|^2_6+|u_t|^2_{2}+1)(\|\nabla u\|^2_1+|D(u)|^2_{\infty}+1).
\end{split}
\end{equation}

Next,  applying $\nabla$ to  $(\ref{eq:1.1})_1$ and multiplying  by $6|\nabla \rho|^{4} \nabla \rho$, one has
\begin{equation}\label{zhu20cccc}
\begin{split}
&(|\nabla \rho|^6)_t+\text{div}(|\nabla \rho|^6 u)+5|\nabla \rho|^6\text{div}u\\
=&-6 |\nabla \rho|^{4}(\nabla \rho)^\top D( u) (\nabla \rho)-6 \rho|\nabla \rho|^{4} \nabla \rho \cdot \nabla \text{div}u.
\end{split}
\end{equation}
It follows from  Lemma \ref{zhenok} and equations $\eqref{dege}_3$ that 
 \begin{equation}\label{zhu55}
\begin{split}
|\nabla^2 u|_6
\leq& C(1+|\nabla u_t|_2+|D(u)|_{\infty}(1+|\nabla \rho|_6)).
\end{split}
\end{equation}
Then integrating (\ref{zhu20cccc}) over $\mathbb{T}^3$ and noticing (\ref{zhu55}), one  immediately obtains
\begin{equation}\label{zhu200}
\begin{split}
\frac{d}{dt}|\nabla \rho|_6
\leq& C|D( u)|_\infty(|\nabla \rho|_6+1)+C(\epsilon)+\epsilon |\nabla u_t|^2_2,
\end{split}
\end{equation}
which, together with \eqref{wzhen5g} and Gronwall's inequaltiy, implies  \eqref{zhu14sshty}.

This completes the proof of this lemma.
\end{proof}

Until now, we have obtained that 
\begin{equation}\label{we11uyy}
\begin{split}
\displaystyle\lim_{T\mapsto \overline{T}} \left(\sup_{0\le t\le T}\big\|\nabla \rho^{\delta-1}(t,\cdot)\big\|_{L^6(\mathbb{T}^3)}+\int_0^T \|D( u)(t,\cdot)\|_ {L^\infty(\mathbb{T}^3)}\ \text{d}t\right)=C_0<\infty,
\end{split}
\end{equation}
for some constant $C_0>0$, then the rest of the proof can be obtained by the completely same argument used in the proof for Lemmas 4.6-4.8.

\bigskip

\section*{Appendix. Some basic lemmas}

In this section, we list  some basic lemmas  to be used later.
The first one is the  well-known Gagliardo-Nirenberg inequality.
\begin{lemma}\cite{oar}\label{lem2as}\
For $p\in [2,6]$, $q\in (1,\infty)$, and $r\in (3,\infty)$, there exists some generic constant $C> 0$ that may depend on $q$ and $r$ such that for
$$f\in H^1(\mathbb{R}^3),\quad \text{and} \quad  g\in L^q(\mathbb{R}^3)\cap D^{1,r}(\mathbb{R}^3),$$
it holds that
\begin{equation}\label{33}
\begin{split}
&|f|^p_p \leq C |f|^{(6-p)/2}_2 |\nabla f|^{(3p-6)/2}_2,\quad |g|_\infty\leq C |g|^{q(r-3)/(3r+q(r-3))}_q |\nabla g|^{3r/(3r+q(r-3))}_r.
\end{split}
\end{equation}
\end{lemma}
Some special cases of this inequality are
\begin{equation}\label{ine}\begin{split}
|u|_6\leq C|u|_{D^1},\quad |u|_{\infty}\leq C|u|^{\frac{1}{2}}_6|\nabla u|^{\frac{1}{2}}_{6}, \quad |u|_{\infty}\leq C\|u\|_{W^{1,r}}.
\end{split}
\end{equation}

The second lemma gives some compactness results obtained via the Aubin-Lions Lemma.
\begin{lemma}\cite{jm}\label{aubin} Let $X_0\subset X\subset X_1$ be three Banach spaces.  Suppose that $X_0$ is compactly embedded in $X$ and $X$ is continuously embedded in $X_1$. Then the following statements hold.

\begin{enumerate}
\item[i)] If $J$ is bounded in $L^p([0,T];X_0)$ for $1\leq p < +\infty$, and $\frac{\partial J}{\partial t}$ is bounded in $L^1([0,T];X_1)$, then $J$ is relatively compact in $L^p([0,T];X)$;\\

\item[ii)] If $J$ is bounded in $L^\infty([0,T];X_0)$  and $\frac{\partial J}{\partial t}$ is bounded in $L^p([0,T];X_1)$ for $p>1$, then $J$ is relatively compact in $C([0,T];X)$.
\end{enumerate}
\end{lemma}

The third  one  can be found in Majda \cite{amj}.
\begin{lemma}\cite{amj}\label{zhen1}
Let  $r$, $a$ and $b$  be constants such that
$$\frac{1}{r}=\frac{1}{a}+\frac{1}{b},\quad \text{and} \quad 1\leq a,\ b, \ r\leq \infty.$$  $ \forall s\geq 1$, if $f, g \in W^{s,a} \cap  W^{s,b}(\mathbb{R}^3)$, then it holds that
\begin{equation}\begin{split}\label{ku11}
&|\nabla^s(fg)-f \nabla^s g|_r\leq C_s\big(|\nabla f|_a |\nabla^{s-1}g|_b+|\nabla^s f|_b|g|_a\big),\\
\end{split}
\end{equation}
\begin{equation}\begin{split}\label{ku22}
&|\nabla^s(fg)-f \nabla^s g|_r\leq C_s\big(|\nabla f|_a |\nabla^{s-1}g|_b+|\nabla^s f|_a|g|_b\big),
\end{split}
\end{equation}
where $C_s> 0$ is a constant depending only on $s$, and $\nabla^s f$ ($s\geq 1$) is the set of  all $\partial^\zeta_x f$  with $|\zeta|=s$. Here $\zeta=(\zeta_1,\zeta_2,\zeta_3)\in \mathbb{R}^3$ is a multi-index.
\end{lemma}

The following lemma is important in the derivation of  the a priori estimates  in Section $3$, which can be found in Remark 1 of \cite{bjr}.
\begin{lemma}\cite{bjr}\label{1}
If $f(t,x)\in L^2([0,T]; L^2)$, then there exists a sequence $s_k$ such that
$$
s_k\rightarrow 0, \quad \text{and}\quad s_k |f(s_k,x)|^2_2\rightarrow 0, \quad \text{as} \quad k\rightarrow+\infty.
$$
\end{lemma}

Next we give one  Sobolev inequalities on the interpolation estimate in the following   lemma.
\begin{lemma}\cite{amj}\label{gag111}
Let  $u\in H^s$, then for any $s'\in[0,s]$,  there exists  a constant $C_s$ only depending on $s$ such that
$$
\|u\|_{s'} \leq C_s \|u\|^{1-\frac{s'}{s}}_0 \|u\|^{\frac{s'}{s}}_s.
$$
\end{lemma}

In order   to improve a weak convergence to the strong convergence, we give the following lemma.
\begin{lemma}\cite{amj}\label{zheng5}
If the function sequence $\{w_n\}^\infty_{n=1}$ converges weakly  to $w$ in a Hilbert space $X$, then it converges strongly to $w$ in $X$ if and only if
$$
\|w\|_X \geq \lim \text{sup}_{n \rightarrow \infty} \|w_n\|_X.
$$
\end{lemma}

The next lemma is used to give the estimate on $\nabla u_t$ in the periodic problem away from the vacuum.
\begin{lemma}\cite{Corn}\label{Korn}
Let $\Omega\subset \mathbb{R}^n$ $(n\geq 2)$ be an open,  connected domain. Then there is a constant $C>0$, known as the Korn constant of $\Omega$, such that, for all vector fields $v=(v^1,..., v^n)\in H^1(\Omega)$,
$$
\|v\|^2_{H^1(\Omega)}\leq C\int_{\Omega} \big(|v|^2+|D(v)|^2\big) \text{\rm d}x.
$$
\end{lemma}

Finally, we give the well-known Fatou's lemma.
\begin{lemma}\label{Fatou}
Given a measure space $(V,\mathcal{F},\nu)$ and a set $X\in \mathcal{F}$, let  $\{f_n\}$ be a sequence of $(\mathcal{F} , \mathcal{B}_{\mathbb{R}_{\geq 0}} )$-measurable non-negative functions $f_n: X\rightarrow [0,\infty]$. Define the function $f: X\rightarrow [0,\infty]$ by setting
$$
f(x)= \liminf_{n\rightarrow \infty} f_n(x),
$$
for every $x\in X$. Then $f$ is $(\mathcal{F},  \mathcal{B}_{\mathbb{R}_{\geq 0}})$-measurable, and   
$$
\int_X f(x) \text{\rm d}\nu \leq \liminf_{n\rightarrow \infty} \int_X f_n(x) \text{\rm d}\nu.
$$
\end{lemma}

\bigskip
\bigskip
\noindent{\bf Acknowledgement:} 
The research of Shengguo Zhu was supported in part by  the Royal Society-- Newton International Fellowships NF170015, and  Monash University-Robert Bartnik Visiting Fellowships.

\bigskip

\noindent{\bf Conflict of Interest:} The authors declare that they have no conflict of
interest.

\end{document}